\pdfoutput=1
\documentclass[11pt]{amsart}
\usepackage[style=alphabetic,maxnames=6]{biblatex}
\usepackage[paperheight=279mm,paperwidth=18cm,textheight=26cm,textwidth=14cm,includehead]{geometry}
\usepackage[T1]{fontenc}
\usepackage[utf8]{inputenc}
\usepackage{amsfonts,amssymb,amsmath,amsthm}
\usepackage{mathtools}
\usepackage{etoolbox}
\usepackage{xcolor} 
\usepackage{esint}
\usepackage{tikz}
\usepackage{bbm,dsfont}
\usepackage{booktabs}

\usepackage[unicode]{hyperref}
\hypersetup{
bookmarks=true,
colorlinks=true,
citecolor=[rgb]{0,0,0.5},
linkcolor=[rgb]{0,0,0.5},
urlcolor=[rgb]{0,0,0.75},
pdfpagemode=UseNone,
pdfstartview=FitH,
pdfdisplaydoctitle=true,
pdflang=en-US
}

\def\PZdefchar#1{
  \expandafter\def\csname frak#1\endcsname{\mathfrak{#1}}
  \expandafter\def\csname bf#1\endcsname{\mathbf{#1}}
  \expandafter\def\csname scr#1\endcsname{\mathcal{#1}}
  \expandafter\def\csname cal#1\endcsname{\mathcal{#1}}}
\def\PZdefloop#1{\ifx#1\PZdefloop\else\PZdefchar#1\expandafter\PZdefloop\fi}
\PZdefloop abcdefghijklmnopqrstuvwxyzABCDEFGHIJKLMNOPQRSTUVWXYZ\PZdefloop

\newcommand{\R}{\mathbb{R}}
\newcommand{\N}{\mathbb{N}}

\newcommand{\dif}{\mathrm{d}}

\newcommand{\dist}{\operatorname{dist}}

\numberwithin{equation}{section}
\newtheorem{theorem}[equation]{Theorem}
\newtheorem{proposition}[equation]{Proposition}
\newtheorem{lemma}[equation]{Lemma}
\newtheorem{corollary}[equation]{Corollary}
\theoremstyle{definition}
\newtheorem{definition}[equation]{Definition}
\newtheorem{hypothesis}[equation]{Hypothesis}

\theoremstyle{remark}
\newtheorem{remark}[equation]{Remark}

\DeclarePairedDelimiter\abs{\lvert}{\rvert}

\DeclarePairedDelimiter\norm{\lVert}{\rVert}

\providecommand\given{}
\newcommand\SetSymbol[1][]{%
\nonscript\:#1\vert
\allowbreak
\nonscript\:
\mathopen{}}
\DeclarePairedDelimiterX\Set[1]\{\}{\renewcommand\given{\SetSymbol[\delimsize]}#1}
\DeclarePairedDelimiterXPP\EE[1]{\E}{\lparen}{\rparen}{}{\renewcommand\given{\SetSymbol[\delimsize]}#1} 

\makeatletter
\newcommand\@avsum[2]{%
  {\sbox0{$\m@th#1\sum$}%
   \vphantom{\usebox0}%
   \ooalign{%
     \hidewidth
     \smash{\vrule height\dimexpr\ht0+1pt\relax depth\dimexpr\dp0+1pt\relax}%
     \hidewidth\cr
     $\m@th#1\sum$\cr
   }%
  }%
}
\newcommand{\avsum}{\mathop{\mathpalette\@avsum\relax}\displaylimits}
\newcommand\@avprod[2]{%
  {\sbox0{$\m@th#1\prod$}%
   \vphantom{\usebox0}%
   \ooalign{%
     \hidewidth
     \smash{\vrule height\dimexpr\ht0+1pt\relax depth\dimexpr\dp0+1pt\relax}%
     \hidewidth\cr
     $\m@th#1\prod$\cr
   }%
  }%
}
\newcommand{\avprod}{\mathop{\mathpalette\@avprod\relax}\displaylimits}
\newcommand{\@avL}[2]{%
\ooalign{{$\m@th#1\mbox{--}$}\cr {$\m@th#1 L$}\cr}}
\newcommand{\avL}{\mathpalette\@avL\relax}
\newcommand{\@avell}[2]{%
\ooalign{{$\m@th#1\mbox{--}$}\cr {$\m@th#1 \ell$}\cr}}
\newcommand{\avell}{\mathpalette\@avell\relax}
\newcommand{\@avD}{%
  \ooalign{{$\mathrm{D}$}\cr \hidewidth\raise.2ex\hbox{$\vert$}\hidewidth\cr}}
\newcommand{\avDec}{\@avD\mathrm{ec}}
\makeatother
\newcommand{\Dec}{\mathrm{Dec}}
\newcommand{\MulDec}{\mathrm{MulDec}}

\DeclareMathOperator{\supp}{supp}
\DeclareMathOperator{\rank}{rank}
\DeclareMathOperator{\diag}{diag}

\newcommand{\ED}{E^{\calD}}
\newcommand{\Part}[2][]{\calP(\ifstrempty{#1}{}{#1,}#2)} 
\newcommand{\Cov}[2][\R^{d+n}]{\calB(#1,#2)} 
\hypersetup{
pdftitle={Decoupling for quadratic surfaces},
pdfauthor={Guo, Zorin-Kranich},
}

\title[Decoupling for quadratic surfaces]{Decoupling for certain quadratic surfaces\\ of low co-dimensions}
\author{Shaoming Guo}
\address[SG]{Department of Mathematics\\ University of Wisconsin Madison\\ USA}
\email{shaomingguo@math.wisc.edu}
\author{Pavel Zorin-Kranich}
\address[PZ]{Mathematical Institute\\ University of Bonn\\ Germany}
\email{pzorin@math.uni-bonn.de}
\subjclass[2010]{42B25 (Primary) 11L15, 26D05 (Secondary)}

\begin{document}
\maketitle
\begin{abstract}
We prove sharp $\ell^{p}L^{p}$ decoupling inequalities for $2$ quadratic forms in $4$ variables.
We also recover several previous results \cite{MR3736493, MR3447712, MR3848437, MR3945730} in a unified way.
\end{abstract}

\section{Introduction}

Let $1\le n\le d$ be integers and $P_1, \dotsc, P_n$ real quadratic forms on $\R^{d}$.
We study decoupling inequalities associated to the quadratic surface
\begin{equation}
\calS = \calS_{d, n}:=\Set{(t, P_1(t), \dotsc, P_{n}(t)) \given t\in [0, 1]^d}.
\end{equation}
For a subset $R\subset [0, 1]^d$, define an extension operator
\begin{equation}\label{extension_operator}
E_R g(x):=\int_R g(t)e^{2 \pi i(t_1 x_1+\dots+t_d x_d+P_1(t)x_{d+1}+\dots P_n(t)x_{d+n})} \dif t,
\quad
x\in\R^{d+n}.
\end{equation}
For a ball $B=B(c_B, r_B)\subset \R^{d+n}$ and $E>0$, define an associated weight
\begin{equation}
w_{B, E}(x):=\Bigl( 1+\frac{\abs{x-c_B}}{r_B} \Bigr)^{-E}.
\end{equation}
Typically, $E$ is a fixed number that is much bigger than $(d+n)$, and will be omitted from the notation $w_{B, E}$.
All implicit constants are allowed to depend on $E$.
For $\delta \in 2^{-\N}$, we will denote by $\Part{\delta}$ the set of all dyadic cubes with side length $\delta$ in $[0,1]^{d}$.

\begin{theorem}\label{thm:main:extension}
Let $1\le n\le 3$ and
\begin{equation}\label{d_n_constraint}
\begin{cases}
d\ge 1 & \text{ if } n=1,\\
2 \leq d \leq 4 & \text{ if } n=2,\\
d=3 & \text{ if } n=3.
\end{cases}
\end{equation}
Assume that for every choice of linearly independent vectors $\vec{w}_{1}, \dotsc, \vec{w}_{d-n}\in \R^d$,
\begin{equation}\label{bl_assumption}
\det [\nabla P_1(t); \dots ; \nabla P_n(t); \vec{w}_{1}; \dots; \vec{w}_{d-n} ]
\not\equiv 0,
\end{equation}
as a polynomial in the variable $t$.
Assume in addition that, for every hyperplane $H\subset \R^d$,
\begin{equation}\label{low_dim_assumption}
\rank \Big(\big(\lambda_1 P_1+\dotsb+\lambda_n P_n\big)|_H\Big)\ge d-2
\end{equation}
for some $\lambda_{1}, \dotsc, \lambda_n\in \R$.

Let $2\le p\le 2+\frac{4n}{d}$, $\epsilon>0$, and $E>0$.
Then, for every locally integrable function $g$, every dyadic number $\delta \in 2^{-\N}$, and every ball $B\subset \R^{d+n}$ of radius $\delta^{-2}$, we have
\begin{equation}\label{eq:main:extension}
\norm{E_{[0, 1]^d}g}_{L^p(w_{B,E})}
\lesssim_{\epsilon,E}
\delta^{-d(\frac{1}{2}-\frac{1}{p})-\epsilon}
\Bigl( \sum_{\Delta \in \Part{\delta}} \norm{ E_{\Delta}g }_{L^p(w_{B,E})}^p \Bigr)^{1/p}.
\end{equation}
\end{theorem}
It would be very interesting to obtain sharp decoupling inequalities for other pairs of $(d, n)$. If one follows the approach in the current paper, one difficulty is in the linear algebra part of the proof, that is, how to have a better understanding of the Brascamp-Lieb transversality conditions. The constraint on $n$ and $d$ in the above theorem comes mainly from the proof of Lemma \ref{181212lem5.7}.

\subsection{Relation to previous work}
Theorem~\ref{thm:main:extension} unifies several previous results, which are summarized in the table below, and provides a new result when $n=2$ and $d=4$, which is a sharp decoupling for a large class of four dimensional quadratic surfaces in $\R^6$.
\begin{center}
\begin{tabular}{ccc}
  \toprule
  $n$ & $d$ & Reference\\
  \midrule
  $1$ & $\geq 1$ & \cite{MR3736493}\\
  $2$ & $2$ & \cite{MR3447712}\\
  $2$ & $3$ & \cite{MR3945730}\\
  $3$ & $3$ & \cite{MR3848437}\\
  \bottomrule
\end{tabular}
\end{center}
The arguments in the above listed papers are quite different from each other.
This point will be elaborated in Section \ref{section:overview}.
Let us first be more precise about how these results can be recovered.
Notice that we use cubes with side length $\delta$, while many articles use cubes with side length $\delta^{1/2}$.

When $n=1$, Bourgain and Demeter in \cite{MR3374964} and \cite{MR3736493} proved \eqref{eq:main:extension} for every (possibly hyperbolic) paraboloid $\Set{(t, P_1(t)) \given t\in [0, 1]^d}$ with non-vanishing Gaussian curvature.
In this case, the implication
\begin{equation}\label{eq:codim1-hypotheses}
\rank (P_{1}) = d
\implies
\eqref{bl_assumption} \text{ and } \eqref{low_dim_assumption}
\end{equation}
can be easily verified, see \cite[Lemma 2.6]{MR3736493}.

When $n=2$ and $d=2$, Bourgain and Demeter \cite{MR3447712} proved \eqref{eq:main:extension} for quadratic forms
\begin{equation}
P_1(t_1, t_2)= A_1 t_1^2+2A_2 t_1 t_2+ A_3 t_2^2,
\quad
P_2(t_1, t_2)= B_1 t_1^2+2B_2 t_1 t_2+ B_3 t_2^2
\end{equation}
under the assumption
\begin{equation}\label{181212e1.9}
\rank
\begin{bmatrix}
A_1, & A_2, & A_3\\
B_1, & B_2, & B_3
\end{bmatrix}
=2.
\end{equation}
Checking \eqref{bl_assumption} amounts to checking
\begin{equation}
\det
\begin{bmatrix}
A_1 t_1+ A_2 t_2, & A_2 t_1+A_3 t_2\\
B_1 t_1+ B_2 t_2, & B_2 t_1+B_3 t_2
\end{bmatrix}
\not\equiv 0,
\end{equation}
which follows immediately from \eqref{181212e1.9}.
The condition \eqref{low_dim_assumption} is vacuously true in this case, since $d-2=0$.

When $d=3$ and $n=2$, Demeter, Shi, and the first author \cite{MR3945730} proved \eqref{eq:main:extension} for two quadratic forms $P_1$ and $P_2$  under the assumption \eqref{bl_assumption} and the assumption that they do not share any common real factor.
Under these assumptions, to verify \eqref{low_dim_assumption}, we just need to notice that $P_1|_H$ and $P_2|_H$ cannot be simultaneously zero.
Let us also mention here that the method used in the current paper significantly simplifies the proof in \cite{MR3945730}.
For an explanation of the major differences, we refer to Section \ref{section:overview}.

When $d=n=3$, Oh \cite{MR3848437} proved \eqref{main_estimate} under the assumption \eqref{bl_assumption} and the assumption that $P_1|_H, P_2|_H$ and $P_3|_H$ do not vanish simultaneously for any hyperplane $H$.
In this case, our assumption \eqref{low_dim_assumption} is just a coordinate-invariant version of Oh's assumptions.

We would also like to point out that a few other sharp decoupling inequalities for quadratic surfaces not covered by Theorem~\ref{thm:main:extension} were proved in \cite{MR3614930, MR3994585,arxiv:1811.02207}.
In many of those cases $n>d$, so that our assumption \eqref{bl_assumption} would not make sense there.

\subsection{Necessity of hypotheses}
Next, let us explain the assumptions \eqref{bl_assumption} and \eqref{low_dim_assumption} in the case $d=4$ and $n=2$.
The assumption \eqref{low_dim_assumption} is a necessary condition for the desired sharp decoupling inequality.
If it is not satisfied, then there exists a hyperplane $H$ such that, for every $\lambda_1$ and $\lambda_2$, we have
\begin{equation}
\rank \Big(\big(\lambda_1 P_1+\lambda_2 P_2\big)|_H\Big)\le 1.
\end{equation}
This further implies that, after changes of variables in $\R^{d}$ and in $\R^{n}$, one can parameterize the restriction of the surface $\calS_{d, n}$ to the plane $H$ as
\begin{equation}
(t_1, t_2, t_3, A t_1^2, 0).
\end{equation}
The $\ell^{p}L^{p}$ decoupling exponent for the parabola $(t_{1},A t_{1})$ is at least $(\frac{1}{2}-\frac{1}{p})$, while the $\ell^{p}L^{p}$ decoupling exponents for the lines $(t_{2})$ and $(t_{3})$ are at least $2(\frac{1}{2}-\frac{1}{p})$, see \eqref{eq:dec-exponent-lower-bd}.
Using tensor products of corresponding examples, one can find a function $g$ that is supported near $H$, such that
\begin{equation}
\norm{E_{[0, 1]^4}g}_{L^p(w_{B^2})} \gtrsim \delta^{-5(\frac 1 2-\frac 1 p)}
\Bigl( \sum_{\Delta \in \Part{\delta}} \norm{E_{\Delta}g}_{L^p(w_{B^2})}^p \Bigr)^{1/p}.
\end{equation}
This violates the desired decoupling inequality \eqref{eq:main:extension}.

We do not know whether the assumption \eqref{bl_assumption} is necessary for the decoupling inequality \eqref{eq:main:extension} to hold.
However, our proof seems to suggest that it is a necessary condition to run the multilinear approach of Bourgain \cite{MR3038558} and Bourgain and Demeter \cite{MR3374964}.
This is indeed the case when $d=3$ and $n=2$, which is the case considered in \cite{MR3945730}.
More precisely, it is proven there that if the condition \eqref{bl_assumption} fails, then no matter how many and which points we pick on the surface, they will never be ``transverse'' in the sense of Definition~\ref{transversality}.

Theorem~\ref{thm:main:extension} applies to an important class of pairs of quadratic forms, namely those pairs of quadratic forms that are simultaneously diagonalizable.
\begin{lemma}\label{lem:simul-diag}
Let $2\le d\le 4$.
Write $t=(t_1, \dots, t_d)$.
Take two quadratic forms
\begin{equation}
P(t)=\sum_{j=1}^d a_j t_j^2 \text{ and } Q(t)=\sum_{j=1}^d b_j t_j^2.
\end{equation}
Assume that
\begin{equation}\label{180713e1.9}
\rank \begin{bmatrix}
a_i & a_j\\
b_i & b_j
\end{bmatrix}
= 2
\qquad\text{ for every } i\neq j.
\end{equation}
Then $P(t)$ and $Q(t)$ satisfy the assumptions of Theorem~\ref{main_theorem} (with $n=2$).
\end{lemma}
The same non-degeneracy condition \eqref{180713e1.9} also appeared in a recent work \cite{MR3652248} by Heath-Brown and Pierce, see Page 95 there. Lemma~\ref{lem:simul-diag} is proved in Appendix~\ref{sec:simul-diag}.

\subsection{Overview of the proof}\label{section:overview}
The proof of Theorem~\ref{eq:main:extension} follows the multilinear approach introduced in \cite{MR2860188} and further developed in \cite{MR3374964,MR3548534,MR3736493,MR3709122,MR3848437,MR3994585,arxiv:1811.02207}.
In Section~\ref{sec:gen}, we formulate this argument for general quadratic surfaces under a lower-dimensional inductive assumption (Hypothesis~\ref{hyp:lower-dim}) and a transversality assumption (Hypothesis~\ref{hyp:transverse}).
In Section~\ref{sec:spec}, we show that these assumptions are satisfied under the conditions \eqref{bl_assumption} and \eqref{low_dim_assumption}.

In the multilinear approach, one uses the Bourgain--Guth argument from \cite{MR2860188} to split the quantity that is to be estimated into a lower-dimensional part and a transversely multilinear part, see Section~\ref{sec:gen:multilinear} and Section~\ref{sec:gen:bourgain-guth}.
The appropriate notion of transversality was introduced in \cite{MR3548534, MR3614930, MR3709122} and is explained in Section~\ref{sec:gen:transverse}.
In Section~\ref{sec:ball-inflation} and Section~\ref{sec:bourgain-demeter}, we run a version of the Bourgain--Demeter iteration argument from \cite{MR3374964} to complete the proof conditionally on Hypothesis~\ref{hyp:lower-dim} and Hypothesis~\ref{hyp:transverse}.

The main new idea in Section~\ref{sec:gen} is the way how lower dimensional contributions are controlled in Section~\ref{sec:gen:lower-dim}.
Specifically, if there is no significant transverse contribution to $E_{[0,1]^{d}}g(x)$, then the main contribution comes from a $1/K$ neighborhood of a low degree subvariety.
In the simplest case, namely when this subvariety is a hyperplane, previous work relied on showing that its $1/K$ neighborhood lies in the $1/K^{2}$ neighborhood of a certain cylinder, see for instance \cite{MR3736493} and \cite{MR3945730}.
This step, if possible, usually involves a large amount of linear algebra calculations, see for instance \cite[Section 4]{MR3945730}.
In the current paper, we show that the step of fitting the $1/K$ neighborhood into the $1/K^{2}$ neighborhood of a cylinder is no longer necessary.
This is the content of Theorem~\ref{thm:near-hyperplane}.
The result for hyperplanes can be extended to graphs with controlled first and second order derivatives by an argument essentially due to Oh \cite{MR3848437}, see Theorem~\ref{thm:near-graph}.

We extend this result to arbitrary subvarieties in Theorem~\ref{thm:dim-reduction-variety}.
In the case of hypersurfaces $\calS$ generated by monomials, this was previously done in \cite{MR3994585,arxiv:1811.02207}.
However, the projection argument in those articles seems to be specific to monomials.
A major difficulty in the general case is how to treat singular points of the subvariety (or, more generally, regions where the curvature is high).
To this end, we cover a neighborhood of the subvariety by neighborhoods of a ``small'' number of graphs (with controlled first and second order derivatives), see Lemma~\ref{lem:sublevel-decomposition-of-variety}.
It is not difficult to imagine that different scales of neighborhoods have to be involved, in order not to use too many graphs.
These scales are called $K_1\ll\dots\ll K_{D+1}$ in Lemma~\ref{lem:sublevel-decomposition-of-variety}.
It is a very interesting phenomenon that in our proof we require $\log K_i\approx_{d, n, \epsilon} \log K_j$ for every $i\neq j$.
In particular, we are not allowed to pick, say, $K_2=2^{K_1}$.
This is important in the iteration in Section~\ref{sec:bourgain-demeter}.

In the end of the overview, let us make a few comments on the differences among proofs in
\cite{MR3374964, MR3736493, MR3447712, MR3945730, MR3848437}. The proofs in \cite{MR3374964, MR3736493} use a $(d+1)$-linear argument, based on the $(d+1)$-linear Loomis-Whitney inequality, while the proofs in \cite{MR3447712, MR3848437} use a bilinear argument, based on certain change of variables, and the proof in \cite{MR3945730} uses an $M$-linear argument with $M$ ranging in an interval of integers (the same as the current paper). The bilinear argument of \cite{MR3447712, MR3848437} is specific to the case $d=n$, that is the dimension of the surface of a half of the dimension of the total space. 

In terms of the Brascamp--Lieb data that are involved: In \cite{MR3945730} the Brascamp--Lieb data that are used are always simple, in the sense that a strict inequality can be achieved for every proper linear subspace $V$.
The Brascamp--Lieb data that appear in \cite{MR3374964, MR3736493} are non-simple.
However, as shown in the current paper (in particular Lemma \ref{lem:BCCT:n=1}), one only needs to use simple data for (hyperbolic) paraboloids.
On a more technical level, this is because the first alternative in Hypothesis~\ref{hyp:transverse} only occurs for the trivial subspace and the full space.
In contrast, in the case $d=n$, one always needs to invoke non-simple Brascamp--Lieb data.

Decoupling theorems are sometimes formulated for functions with Fourier support in $\calS$.
However, in order to use a lower-dimensional inductive assumption such as Hypothesis~\ref{hyp:lower-dim}, one needs a version of the decoupling inequality that holds for functions with Fourier support in a $\delta^{2}$-neighborhood of the surface $\calS$.
We find it convenient to use such a more general version throughout, thus avoiding some technical computations as e.g.\ in \cite[Section 5]{MR3592159}.
This more general form of the decoupling inequality is explained in Section~\ref{sec:fourier-support}.

\subsection{Relaxed Fourier support restriction}\label{sec:fourier-support}
In this section we formulate a decoupling inequality for functions with Fourier support in a neighborhood of the surface $\calS$.
Informally speaking, for every $\theta\in\Part{\delta}$, we consider functions $f_{\theta}$ whose Fourier support is inside a parallelepiped that contains the graph of $P$ over $\theta$ and is close to being smallest among all parallelepipeds with this property, up to a multiplicative factor.
This is illustrated in Figure~\ref{fig:Fourier-support} in the case $d=n=1$.

\begin{figure}
\newcommand{\figparabolicscaling}[1]{
\begin{scope}[black,cm={0.5,0,0,0.25,(0,0)}]
#1
\end{scope}
\begin{scope}[black,cm={0.5,0.5,0,0.25,(0.5,0.25)}]
#1
\end{scope}
}
\newcommand{\figunitball}{\draw (-1.5,-1.5) rectangle (1.5,1.5);}
\begin{center}
\begin{tabular}{cc}
\begin{tikzpicture}
\figunitball
\figparabolicscaling{\figunitball}
\figparabolicscaling{\figparabolicscaling{\figunitball}}
\end{tikzpicture} &
\begin{tikzpicture}
\figunitball
\figparabolicscaling{\figparabolicscaling{\figparabolicscaling{\figunitball}}}
\end{tikzpicture}
\end{tabular}
\end{center}
\caption{$\supp \widehat{f_{\theta}}$, $\theta \in \Part{\delta}$, in the case $d=n=1$.
  On the left, $\delta \in \Set{2^{0},2^{-2},2^{-2}}$.
  On the right, $\delta \in \Set{2^{0}, 2^{-3}}$.}
\label{fig:Fourier-support}
\end{figure}
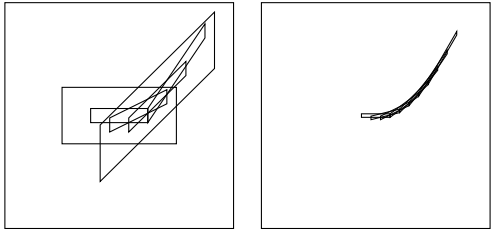
We proceed with a formal definition.
For $\theta = a + \delta [0,1]^{d} \in \Part{\delta}$, let
\[
L_{\theta} :=
\begin{pmatrix}
\delta I_{d} & 0\\
0 & \delta^{2} I_{n}
\end{pmatrix}
\begin{pmatrix}
I_{d} & \nabla P(a)\\
0 & I_{n}
\end{pmatrix},
\text{ where }
\nabla P(a)
=
\begin{pmatrix}
\partial_{1} P_{1}(a) & \dots & \partial_{1} P_{n}(a)\\
\vdots & & \vdots \\
\partial_{d} P_{1}(a) & \dots & \partial_{d} P_{n}(a)
\end{pmatrix},
\]
and $I_{d}$ denotes the identity $d\times d$ matrix.
Let
\begin{equation}
\label{eq:Uncertainty-box}
\calU_{\theta} := (a,P(a)) + L_{\theta}^{*}([-C,C]^{d+n}),
\end{equation}
where $C$ is a large constant to be chosen in a moment.
The set $\calU_{\theta}$ is a parallelepiped whose projection onto $\R^{d}$ is a cube containing $\theta$ and that contains the graph of $P$ restricted to $\theta$.
We choose $C$ large enough, depending on the size of coefficients of $P$, to make these parallelepipeds nested, in the sense that
\begin{equation}
\label{eq:Uncertainty-box-nested}
2\theta \subseteq 2\theta'
\implies
\calU_{\theta} \subseteq \calU_{\theta'}.
\end{equation}
We will denote by $f_{\theta}$ an arbitrary function with $\supp \widehat{f_{\theta}} \subseteq \calU_{\theta}$.
In other words, $f_{\theta}$ is of the form $M_{\theta}f$, where $f$ is an arbitrary function on $\R^{d+n}$ with $\supp \widehat{f} \subset [-C,C]^{d+n}$, and
\[
M_{\theta}f(x,y) = e(a\cdot x + P(a) \cdot y) (f \circ L_{\theta})(x,y),
\quad
x\in\R^{d}, y\in\R^{n}.
\]
Here, $x, y$ and $(x, y)$ are treated as row vectors.

For each $\delta>0$, let the \emph{decoupling constant} $\Dec^{p}(\delta)$ be the smallest constant such that the inequality
\begin{equation}\label{eq:dec-const}
\norm[\big]{\sum_{\theta \in \Part{\delta}} f_{\theta}}_{L^p(\R^{d+n})}
\le
\Dec^{p}(\delta) (\sum_{\theta} \norm{f_{\theta}}_{L^p(\R^{d+n})}^p)^{1/p}
\end{equation}
holds for arbitrary functions $f_{\theta}$ with $\supp \widehat{f_{\theta}} \subseteq \calU_{\theta}$.

The decoupling constant depends on $d,n$ and $P_{1},\dotsc,P_{n}$, and we will sometimes indicate this dependence by subscripts when several different decoupling constants are involved.
We will also omit the exponent $p$ when there is only one such exponent involved.

\begin{remark}
The decoupling constant \eqref{eq:dec-const} also depends, in a monotonically increasing way, on the Fourier support parameter $C$.
This dependence is entirely harmless, as for dyadic $C'$ the decoupling constant at scale $\delta$ with parameter $C'C$ can be easily controlled by the decoupling constant at scale $C'\delta$ with parameter $C$.
The only important thing about the parameter $C$ is that it has to be kept constant throughout various inductive procedures.
\end{remark}

The operators $M_{\theta}$ come from an action of the group of transformations generated by translations and dilations of $\R^{d}$.
This makes parabolic scaling easy, that is, one can apply an affine change of variables in the Fourier space and prove 
\begin{lemma}[Parabolic scaling]
\label{lem:parabolic-scaling}
Under the above notation, we have 
\[\norm[\big]{\sum_{\theta \in \Part[Q]{\delta}} f_{\theta}}_{L^p}
\le
\Dec^{p}(\delta/\sigma) (\sum_{\theta} \norm{f_{\theta}}_{L^p}^p)^{1/p}
\]
for any dyadic numbers $0<\delta\leq\sigma\leq 1$ and any $Q \in \Part{\sigma}$.
\end{lemma}
\begin{proof}
This follows from the fact, for $\theta \in \Part[Q]{\delta}$, we have $M_{\theta} = M_{Q}M_{\theta'}$ for a suitable $\theta' \in \Part{\delta/\sigma}$.
\end{proof}
\begin{remark}[Local decoupling]\label{rem:global-to-local}
Let $\eta$ be a positive Schwartz function on $\R^{d+n}$ such that $\supp \hat{\eta} \subset B(0,c)$ and $\eta \geq 1$ on $B(0,1)$.
Let $B \subset \R^{d+n}$ be a ball of radius $\delta^{-2}$.
Then, applying \eqref{eq:dec-const} with the Fourier support parameter $C$ replaced by $C+c$ to functions $f_{\theta}\eta_{B}$, where $\eta_{B} := \eta(\delta^{2}(\cdot-c(B)))$, we obtain
\[
\norm[\big]{\sum_{\theta \in \Part{\delta}} f_{\theta}}_{L^p(B)}
\leq
\Dec^{p}(\delta) (\sum_{\theta} \norm{\eta_{B} f_{\theta}}_{L^p}^p)^{1/p}.
\]
By \cite[Section 4]{MR3592159}, this implies the localized estimate
\[
\norm[\big]{\sum_{\theta \in \Part{\delta}} f_{\theta}}_{L^p(w_{B})}
\lesssim
\Dec^{p}(\delta) (\sum_{\theta} \norm{f_{\theta}}_{L^p(w_{B})}^p)^{1/p}.
\]
Similarly, we can localize the rescaled decoupling inequality in Lemma ~\ref{lem:parabolic-scaling}.
In fact, we can localize that inequality further to ellipsoids of dimensions $\delta^{-2}\sigma$ ($d$ times) $\times \delta^{-2}$ ($n$ times), but this will not be necessary.
\end{remark}

By Remark~\ref{rem:global-to-local}, Theorem~\ref{thm:main:extension} will follow from the next result.
\begin{theorem}\label{main_theorem}
Let $1\le n\le 3$.
Assume \eqref{d_n_constraint}, \eqref{bl_assumption}, and \eqref{low_dim_assumption}.
Then
\begin{equation}\label{main_estimate}
\Dec^{p}(\delta)\lesssim_{\epsilon} \delta^{-d(\frac{1}{2}-\frac 1 p)-\epsilon}
\end{equation}
for every $2\le p\le 2+\frac{4n}{d}$ and every $\epsilon>0$.
\end{theorem}
Theorem~\ref{main_theorem} will in turn follow directly from Theorem~\ref{thm:dec:general}, once the hypotheses of the latter result are verified in Section~\ref{sec:spec}.

\subsection{Sharpness of the exponents}\label{sec:sharpness}
We recall standard examples that show that, for $2 \leq p,q < \infty$, the $\ell^{q}L^{p}$ decoupling inequality
\[
\norm[\big]{\sum_{\theta \in \Part{\delta}} f_{\theta}}_{L^p}
\lesssim
\delta^{-\Lambda} (\sum_{\theta} \norm{f_{\theta}}_{L^p}^q)^{1/q}
\]
can only hold if
\begin{equation}\label{eq:dec-exponent-lower-bd}
\Lambda \geq \max \Bigl( d-\frac{d}{q}-\frac{d+2n}{p}, d\bigl(\frac12-\frac1q\bigr) \Bigr).
\end{equation}
Consider first $f_{\theta} = M_{\theta} f$, where $f$ is a fixed Schwartz function with $\hat{f}$ positive and compactly supported.
Then by scaling $\norm{f_{\theta}}_{p} \sim \delta^{-(d+2n)/p}$.
On the other hand, $\sum_{\theta} f_{\theta} \gtrsim \delta^{-d}$ on a fixed neighborhood of $0$.
It follows that $\Lambda \geq d-\frac{d}{q}-\frac{d+2n}{p}$.

Consider next $f_{\theta}(x)=\eta(\delta^{2} x) e^{2 \pi i c_{\theta} \cdot x}$, where $\eta$ is a Schwartz function with $\hat\eta$ compactly supported and $c_{\theta}$ is a point on the surface $\calS$ over $\theta$.
Then $\norm{f_{\theta}}_{p} \sim \delta^{-2(d+n)/p}$ and by H\"older's inequality and orthogonality
\begin{multline*}
\delta^{-2(d+n)(\frac12-\frac1p)} \norm{\sum_{\theta} f_{\theta}}_{p}
\sim
\norm{\eta(\delta^{2}\cdot)}_{\frac{1}{1/2-1/p}} \norm{\sum_{\theta} f_{\theta}}_{p}
\geq
\norm{\sum_{\theta} \eta(\delta^{2}\cdot) f_{\theta}}_{2}\\
\gtrsim
\bigl( \sum_{\theta} \norm{\eta(\delta^{2}\cdot) f_{\theta}}_{2}^{2} \bigr)^{1/2}
\sim
\delta^{-d/2}\delta^{-2(d+n)/2}.
\end{multline*}
It follows that $\Lambda \geq d\bigl(\frac12-\frac1q\bigr)$.

\begin{remark}
It is known from \cite[p.~118]{MR1209299} that the $\epsilon$ loss in \eqref{main_estimate} cannot be completely removed in general.
\end{remark}

\subsection*{Notation}
$\Part[Q]{\delta}$ is the partition of a dyadic cube $Q$ into dyadic cubes with side length $\delta$.
We omit $Q$ if $Q=[0,1]^{d}$.

We use $C$ to denote a large constant that is allowed to change from line to line.
Its precise value is of no relevance. The letter $\Lambda$ is used throughout the paper, see Hypothesis \ref{hyp:lower-dim} for its meaning. 

For a sequence of real numbers $\{A_i\}_{i=1}^M$, we abbreviate $\avprod A_{i} := \bigl(\prod_{i=1}^{M} A_{i}\bigr)^{1/M}$.
Also, we define averaged integrals: 
\[
\norm{f}_{\avL^{p}(B)}:=(\frac{1}{|B|}\int_B |f|^p )^{1/p} \text{ and } \norm{f}_{\avL^{p}(w_B)}:=(\frac{1}{|B|}\int |f|^p w_B)^{1/p}.
\]

For $\sigma>0$ and $E\subset \R^d$, we will use $N_{\sigma}(E)$ to denote the $\sigma$-neighborhood of the set $E$.

For a non-negative number $a$, we will use $\lfloor a\rfloor$ to denote the greatest integer less than or equal to $a$, and $\lceil a\rceil$ to denote the least integer greater than or equal to $a$.

\subsection*{Acknowledgment}
S.G.\ was supported in part by a direct grant for research from the Chinese University of Hong Kong (4053295).
P.Z.\ was partially supported by the Hausdorff Center for Mathematics in Bonn (DFG EXC 2047).
He would also like to thank Po-Lam Yung for inviting him to the Chinese University of Hong Kong, where part of this work was conducted. Both authors thank the referee for reading the paper carefully and providing useful feedbacks that improved the exposition of the paper.

\section{General surfaces}\label{sec:gen}

\subsection{Lower dimensional decoupling}\label{sec:gen:lower-dim}

Let $\calH\subset \R^d$ be a hyper-plane that intersects $[0, 1]^d$.
Without loss of generality, we write it as a graph
\begin{equation}\label{181212e4.2}
t_d=\calL(t') \text{ where } t'=(t_1, \dots, t_{d-1}),
\end{equation}
and $\calL(t')$ is a linear form of $t'$ with $\abs{\nabla \calL}\lesssim 1$.
Consider the new quadratic forms $P_{j}'(t') := P_{j}(t',\calL(t'))$, $j=1,\dotsc,n$, and define the associated decoupling constant $\Dec_{\calH}^{p}(\delta)$ analogously to \eqref{eq:dec-const}.
Moreover, the index $p$ will be dropped from the notation $\Dec_{\calH}^{p}(\delta)$ whenever it is clear from the context which $p$ we are using. 

\begin{theorem}\label{thm:near-hyperplane}
Suppose that $\Dec_{\calH}^{p}(\delta) \lesssim \delta^{-\Lambda}$.
Then, for every $c<\infty$ and $\delta \in 2^{-\N}$, we have
\begin{equation}
\label{eq:hyperplane-decoupling}
\norm[\Big]{\sum_{\substack{\Box \in \Part{\delta},\\ c\Box\cap \calH\neq \emptyset}} f_{\Box} }_{p}
\lesssim_c
\abs{\log \delta}^{C_{\calS}} \delta^{-\Lambda}
\Big(\sum_{\Box} \norm{f_{\Box}}^p_{p}\Big)^{1/p},
\end{equation}
for a constant $C_{\calS}\lesssim 1$ that depends only on the surface $\calS$.
\end{theorem}

In previous work of Bourgain and Demeter \cite[Section 2]{MR3736493} and of Demeter, Shi, and the first author \cite[Section 4]{MR3945730}, similar results were obtained in certain special cases ($d$ arbitrary, $n=1$ and $d=3,n=2$, respectively).
In the language of the proof below, the idea was to make a projection after which the Fourier support fits into an $O(\delta^{2})$ neighborhood of some lower dimensional surface.
In this situation, one can obtain \eqref{eq:hyperplane-decoupling} by applying the lower-dimensional decoupling fiberwise at scale $\delta$.
Our proof shows that, using an additional induction on scales, \eqref{eq:hyperplane-decoupling} can be obtained without investigating the ``geometry'' of the graph of $P$.

\begin{proof}[Proof of Theorem~\ref{thm:near-hyperplane}.]
Using \eqref{eq:Uncertainty-box-nested}, we assign each $\Box$ with $c\Box \cap \calH \neq \emptyset$ to a cube $\tilde{\Box}$ of side length $C'\delta$ such that $\tilde{\Box} \cap \calH \neq \emptyset$, $\calU_{\Box} \subseteq \calU_{\tilde{\Box}}$, and $C'$ depends only on $c$.
Since we assign boundedly many $\Box$ to each $\tilde{\Box}$, we may assume $c=1$.

For notational simplicity, we assume $\calH = \Set{ t \given t_{d}=0 }$.
Let $P'(t'):=P(t',0)$.

Let $A(\delta)$ be the smallest constant for which the inequality
\[
\norm[\Big]{\sum_{\substack{\Box \in \Part{\delta},\\ \Box\cap \calH\neq \emptyset}} f_{\Box} }_{L^p}
\leq A(\delta)
\Big( \sum_{\Box} \norm{f_{\Box}}^p_{L^p}\Big)^{1/p}
\]
holds.
It is easy to see that $A(\delta) \lesssim \delta^{-C}$ for some large $C$.

Let $\Box \in \Part{\delta}$ with $\Box \cap \calH \neq \emptyset$ and $\xi = (\xi',\xi_{d}) \in \Box$. 
Then, since $\abs{\nabla P} \lesssim 1$ on $B(0,C)$, the projection of $L_{\Box}^{*}([-C,C]^{d+n})$ onto $\R^{d-1}\times \R^{n}$ is contained in a $O(\delta)$-neighborhood of $L_{\Box'}^{*}([-C,C]^{d-1+n})$, where $\Box'$ is the projection of $\Box$ onto $\R^{d-1}$.

For each fixed $x_{d}$, this gives a restriction on the fiberwise Fourier support restriction $\widehat{f_{\Box}(\cdot,x_{d},\cdot)}$ that is not strong enough to apply decoupling at scale $\delta$, but is sufficient to apply decoupling at scale $\delta^{1/2}$.
Hence, we obtain
\begin{multline*}
\norm[\Big]{\sum_{\substack{\Box \in \Part{\delta},\\ \Box\cap \calH\neq \emptyset}} f_{\Box} }_{L^p(\R^{d-1} \times \Set{x_{d}} \times \R^{n})}\\
\lesssim
\Dec_{\calH}(c\delta^{1/2})
\Big( \sum_{\substack{\Box'\in\Part{\delta^{1/2}}, \\ \Box'\cap \calH\neq \emptyset}}\norm{ \sum_{\substack{\Box \in \Part[\Box']{\delta},\\ \Box\cap \calH\neq \emptyset}} f_{\Box} }^p_{L^p(\R^{d-1} \times \Set{x_{d}} \times \R^{n})}\Big)^{1/p}
\end{multline*}
for every $x_{d}$.
Integrating in $x_{d}$, we obtain
\[
\norm[\Big]{\sum_{\substack{\Box \in \Part{\delta},\\ \Box\cap C\calH\neq \emptyset}} f_{\Box} }_{p}
\lesssim
\delta^{-\Lambda/2}
\Big( \sum_{\substack{\Box'\in\Part{\delta^{1/2}}, \\ \Box'\cap \calH\neq \emptyset}}\norm{ \sum_{\substack{\Box \in \Part[\Box']{\delta},\\ \Box\cap \calH\neq \emptyset}} f_{\Box} }^p_{p}\Big)^{1/p}.
\]
By Lemma~\ref{lem:parabolic-scaling}, we have
\[
\norm{ \sum_{\substack{\Box \in \Part[\Box']{\delta},\\ \Box\cap \calH\neq \emptyset}} f_{\Box} }_{p}
\leq A(\delta^{1/2})
\Big( \sum_{\substack{\Box \in \Part[\Box']{\delta},\\ \Box\cap \calH\neq \emptyset}} \norm{ f_{\Box} }_{p}^{p} \Big)^{1/p}.
\]
It follows that
\[
A(\delta) \lesssim \delta^{-\Lambda/2} A(\delta^{1/2}).
\]
Iterating this inequality approximately $\log\log \frac{1}{\delta}$ times, we obtain the claim.
\end{proof}

Next, we will prove a version of Theorem~\ref{thm:near-hyperplane} for curved hypersurfaces.
\begin{hypothesis}\label{hyp:lower-dim}
Suppose that, for every hyperplane $\calH \subset \R^{d}$ passing through $0$, we have $\Dec_{\calH}^{p}(\delta) \lesssim \delta^{-\Lambda}$, uniformly in $\calH$.
\end{hypothesis}

By an affine change of coordinates in the Fourier space similar to Lemma~\ref{lem:parabolic-scaling}, Hypothesis \ref{hyp:lower-dim} implies the superficially stronger statement that  $\Dec_{\calH}^{p}(\delta) \lesssim \delta^{-\Lambda}$, uniformly in all hyperplanes $\calH$ that have non-empty intersection with the unit cube.

In the situation of Theorem~\ref{main_theorem}, Hypothesis~\ref{hyp:lower-dim} will be verified in Section~\ref{sec:spec:lower-dim} for an appropriate exponent $\Lambda$, depending on $d$ and $n$.

\begin{theorem}\label{thm:near-graph}
Let $2\leq p < \infty$ and assume Hypothesis~\ref{hyp:lower-dim}.
Then, for every $\epsilon>0$, every $\tilde{C} < \infty$, and every hypersurface $\widetilde{\calH}\subset [0,1]^d$ that can be written as a graph
\begin{equation}\label{181126e4.9}
t_d=\calL(t') \text{ with }
\norm{\calL}_{C^{2}} \leq \tilde{C},
\end{equation}
we have
\begin{equation}\label{181209e4.11}
\norm[\Big]{\sum_{\substack{\Box\in\Part{\delta},\\ \Box\cap \widetilde{\calH}\neq \emptyset}}f_{\Box} }_{p}
\lesssim_{\epsilon,\tilde{C}}
\delta^{-\Lambda-\epsilon}
\Big( \sum_{\Box}\norm{f_{\Box} }^p_{p}\Big)^{1/p}
\end{equation}
for every $0<\delta\le 1$.
The implicit constant in \eqref{181209e4.11} may depend on the bound for the $C^{2}$ norm in \eqref{181126e4.9}, but not otherwise on $\widetilde{\calH}$.
\end{theorem}

\begin{proof}[Proof of Theorem~\ref{thm:near-graph}.]
The proof is via an iteration argument, essentially due to Oh \cite{MR3848437}.
It is also closely related to the iteration argument of Pramanik and Seeger \cite{MR2288738}.

All implicit constants in this proof are allowed to depend on the constant $\tilde{C}$ in \eqref{181126e4.9}.
For $\kappa \leq 2^{-10}$, let $A(\kappa,\delta)$ be the smallest constant such that the inequality
\[
\norm[\Big]{\sum_{\substack{\Box\in\Part{\delta},\\ \Box\cap \widetilde{\calH}\neq \emptyset}}f_{\Box} }_{p}
\leq
A(\kappa,\delta)
\Big( \sum_{\Box}\norm{f_{\Box} }^p_{p}\Big)^{1/p}
\]
holds for all hypersurfaces $\widetilde{\calH}$ that are parameterized by functions $\calL$ with
\begin{equation}
\label{eq:6}
\abs{\nabla\calL} \leq \tilde{C}
\quad\text{and}\quad
\abs{\bfH\calL} \leq 2^{10}\tilde{C}\kappa,
\end{equation}
where $\bfH\calL$ denotes the Hessian matrix of second derivatives of $\calL$.
Then the constant in \eqref{181209e4.11} is bounded by $A(2^{-10},\delta)$.

Suppose that \eqref{eq:6} holds.
Let $\calH$ be the tangent plane at some point of $\widetilde{\calH}$.
Then $\widetilde{\calH}$ is contained in the $O(\kappa)$-neighborhood of $\calH$.
If $\kappa \leq \delta$, then we can apply Theorem~\ref{thm:near-hyperplane} and obtain
\begin{equation}
\label{eq:5}
A(\kappa,\delta) \lesssim \abs{\log \delta}^{C} \delta^{-\Lambda}.
\end{equation}

If $\kappa > \delta$, then we can instead apply Theorem~\ref{thm:near-hyperplane} at scale $\kappa$.
This gives
\begin{equation}
\label{eq:4}
\norm[\Big]{\sum_{\substack{\Box\in\Part{\delta},\\ \Box\cap \widetilde{\calH}\neq \emptyset}}f_{\Box} }_{p}
\lesssim
\abs{\log\kappa}^{C} \kappa^{-\Lambda}
\Big( \sum_{\substack{\Box' \in \Part{\kappa},\\ 2\Box' \cap \calH \neq \emptyset}} \norm{ \sum_{\Box \subset \Box'} f_{\Box} }^p_{p}\Big)^{1/p}.
\end{equation}
The crucial observation now is that, after rescaling any of the $\Box'$ to unit scale, the surface $\widetilde{\calH} \cap \Box'$ becomes parameterized by a function with first derivative still bounded by $\tilde{C}$, and the second derivative bounded by $2^{10} \tilde{C} \kappa^{2}$.
By Lemma~\ref{lem:parabolic-scaling}, it follows that
\begin{equation}
\label{eq:3}
\norm{ \sum_{\substack{\Box\in\Part[\Box']{\delta},\\ \Box\cap \widetilde{\calH}\neq \emptyset}} f_{\Box} }_{p}
\leq
A(\kappa^{2},\delta/\kappa)
\Big( \sum_{\Box} \norm{ f_{\Box} }^p_{p}\Big)^{1/p}.
\end{equation}
Summing boundedly many copies of this estimate, we may replace the restriction $\Box \cap \widetilde{\calH} \neq \emptyset$ by $2\Box \cap \widetilde{\calH} \neq \emptyset$ on the left-hand side.
Inserting a rescaled version of \eqref{eq:3} into \eqref{eq:4}, we obtain
\[
A(\kappa,\delta)
\lesssim
\abs{\log\kappa}^{C} \kappa^{-\Lambda} A(\kappa^{2},\delta/\kappa).
\]
Starting with $\kappa=2^{-10}$ and applying this inequality at most approximately $\log \log \frac{1}{\delta}$ times, we arrive in the situation $\kappa \leq \delta$, because $\kappa$ is squared in each step.
In the end, we apply \eqref{eq:5}.
\end{proof}

In the Bourgain--Guth iteration scheme that is used to prove the equivalence between linear and multilinear decouplings, we have to apply lower dimensional decoupling to families of functions with Fourier support close to a subvariety.
In order to apply Theorem~\ref{thm:near-graph}, we will cover a neighborhood of the subvariety by neighborhoods of hypersurfaces with controlled curvature.
To this end, the following fact will be useful.

\begin{lemma}\label{lem:sublevel-f-vs-Df}
Let $\Omega \subset \R^{n}$ be open and $f \in C^{2}(\Omega) \cap C(\bar\Omega)$ with $\abs{\bfH f} \leq 1$.
Then, for all $\sigma,\eta>0$, we have
\begin{equation}\label{eq:sublevel-f-vs-Df}
\Set{ \abs{f} \leq \sigma\eta}
\subset
\Set{ \abs{\nabla f} \leq \sigma+\eta }
\cup
N_{\sigma} \bigl( \partial\Omega \cup ( \Set{f=0} \cap \Set{\abs{\nabla f} > \eta} ) \bigr).
\end{equation}
\end{lemma}
\begin{proof}
By Taylor's formula and the intermediate value theorem, for every $x\in\Omega$ and $t>0$ the inequality
\[
\abs{f(x)} - t \abs{\nabla f(x)} + t^{2} \norm{\bfH f}_{\infty}/2 \leq 0
\]
implies $\dist(x, \partial\Omega \cup \Set{f = 0}) \leq t$.
In the case $\abs{f(x)} \leq \sigma\eta$ and $\abs{\nabla f(x)} > \sigma+\eta$, the above inequality holds with $t=\sigma$.
Moreover, if $B(x,\sigma) \subset \Omega$, then $\abs{\nabla f}>\eta$ on that ball.
\end{proof}

\begin{lemma}\label{lem:sublevel-decomposition-of-variety}
For every natural numbers $n,A \geq 1$, every $D\geq 0$, and every sufficiently large $K>1$, there exist
\[
K_1\leq K_2\leq \dotsb \leq K_{D+1} = K
\text{ with }
K_{j+1} \sim_{n,j} K_{j}^{A+1}
\]
such that, for every normalized polynomial $P$ of degree $D$ in $n$ variables with real coefficients, there exists an increasing sequence of multiindices $\alpha_{D} <  \alpha_{D-1} < \dotsb <  \alpha_{1}$ with $\abs{\alpha_{j}} = D-j$ such that
\begin{equation}
\label{eq:sublevel-decomposition-of-variety}
\Set{ \abs{P} < 1/K_{D+1} } \cap B(0,1)
\subseteq
\bigcup_{j=1}^{D} N_{1/K_{j}^{A}} \Bigl( Z_{\partial^{\alpha_{j}}P} \cap \Set{\abs{\nabla \partial^{\alpha_{j}} P} \geq 1/K_{j}} \Bigr).
\end{equation}
Here we say that $P$ is a normalized polynomial if $\norm{P}=1$, where $\norm{P}$ is the $\ell^1$ sum of the coefficient of $P$.
Also, $Z_{P} = \Set{ x \given P(x)=0 }$.
\end{lemma}
\begin{proof}
By induction on $D$.
In the case $D=0$, the left hand side of \eqref{eq:sublevel-decomposition-of-variety} is empty provided $K>1$, since $P\equiv 1$.

Suppose now that $D \geq 1$, and that the conclusion is already known with $D$ replaced by $D-1$.

If $\norm{\partial^{\alpha} P} \le c_{1}$ for some sufficiently small constant $c_{1} = c_{1}(n,D) > 0$ and for all $\abs{\alpha}=1$, then, since $P$ is normalized, the left hand side of \eqref{eq:sublevel-decomposition-of-variety} is empty provided that $K_{D+1}$ is large enough.
Hence, we may assume $\norm{\partial^{\alpha} P} \geq c_{1}$ for some multiindex $\alpha$ with $\abs{\alpha}=1$.

Since $P$ is normalized, we have $c_{0}\abs{\bfH P(x)} \leq 1$ for some $c_{0} = c_{0}(n,D) > 0$ and all $x\in \Omega := B(0,2)$.
By Lemma~\ref{lem:sublevel-f-vs-Df}, we obtain
\begin{multline*}
\Set{ \abs{P} < 1/K_{D+1}}
\subset
\Set{ \abs{c_{0} P} < c_{0}/K_{D}^{A+1} }\\
\subset
\Set{ \abs{c_{0}\nabla P} < 1/K_{D}^{A}+c_{0}/K_{D} }
\cup
N_{1/K_{D}^{A}} \bigl( \partial\Omega \cup (Z_{P} \cap \Set{\abs{c_{0} \nabla P} > c_{0}/K_{D}}) \bigr)
\end{multline*}
provided $c_{0}/K_{D+1} \leq c_{0}/K_{D}^{A+1}$.

Since $K_{D}\geq 1$, the above neighborhood of $\partial\Omega$ does not intersect $B(0,1)$, and we obtain
\begin{multline*}
\Set{ \abs{P} < 1/K_{D+1}} \cap B(0,1)
\subset
\Set{ \abs{c_{0}\nabla P} < 1/K_{D}^{A}+c_{0}/K_{D} }\\
\cup
N_{1/K_{D}^{A}} (Z_{P} \cap \Set{\abs{\nabla P} > 1/K_{D}}).
\end{multline*}
The second term is of the required form.
In the first term, we apply the inductive hypothesis with $D$ replaced by $D-1$, $P$ replaced by $\|\partial^{\alpha} P\|^{-1} \partial^{\alpha} P$, and $K$ replaced by $\tilde{K}$ satisfying $1/K_{D}+1/(c_{0}K_{D}^{A}) = c_{1}/\tilde{K}$.
\end{proof}

The next result extends \cite[Claim 5.10]{MR3709122} and \cite[Proposition 4.1]{MR3848437}.
\begin{theorem}\label{thm:dim-reduction-variety}
Assume Hypothesis~\ref{hyp:lower-dim}.
For every $D\geq 1$ and $A> 1$, for every sufficiently large $K$, there exist
\[
K_1\leq K_2\leq \dotsb \leq K_{D+1} = K
\text{ with }
K_{j+1} \sim_{n,j} K_{j}^{A+1}
\]
such that, for every non-zero polynomial $P$ of degree $D$, there exist collections of pairwise disjoint cubes $\calG_{j} \subset \Part{1/K_{j}^{A}}$, $j=1,\dotsc,D$, such that
\[
N_{1/K}(Z_{P}) \cap [0,1]^{d}
\subset
\bigcup_{j=1}^{D} \bigcup_{\Box \in \calG_{j}} \Box
\]
and
\begin{equation}
\norm[\Big]{ \sum_{\Box \in \calG_{j}} f_{\Box} }_{p}
\lesssim_{D,\epsilon}
K_{j}^{d}(K_{j}^{A-1})^{\Lambda+\epsilon} \Big( \sum_{\Box \in \calG_{j}} \norm{f_{\Box}}^p_{p}\Big)^{1/p}.
\end{equation}
\end{theorem}
\begin{proof}
Let $P_{j} := \partial^{\alpha_{j}}P$ be as in Lemma~\ref{lem:sublevel-decomposition-of-variety} and
\[
Z_{j} := Z_{P_{j}} \cap \Set{\abs{\nabla P_{j}} \geq 1/K_{j}}.
\]
Let
\[
\calG_{j} := \Set{\Box \in \Part{1/K_{j}^{A}} \given C\Box \cap Z_{j} \neq \emptyset} \setminus \bigcup_{j'<j} \bigcup_{\Box' \in \calG_{j'}} \Part[\Box']{1/K_{j}^{A}}.
\]
Using Minkowski's inequality at scale $1/(C K_{j})$, it suffices to show
\[
\norm[\Big]{ \sum_{\Box \in \calG_{j} : \Box \subset Q} f_{\Box} }_{p}
\lesssim_{\epsilon, D}
(K_{j}^{A-1})^{\Lambda+\epsilon} \Big( \sum_{\Box} \norm{f_{\Box}}^p_{p}\Big)^{1/p}
\]
for every $Q \in \Part{1/(CK_{j})}$.
But if there exists $\Box \in \calG_{j}$ with $\Box \subset Q$, then $\abs{\nabla P_{j}} \gtrsim 1/K_{j}$ on $CQ$, so by the implicit function theorem $Z_{j} \cap CQ$ is a hypersurface with curvature $\lesssim K_{j}$.
After scaling $Q$ to the unit scale, the set $Z_{j} \cap CQ$ becomes a graph with curvature $\lesssim 1$, and the claim follows by Theorem~\ref{thm:near-graph}.
\end{proof}

For a fixed $d$, we can choose an $A$ in Theorem~\ref{thm:dim-reduction-variety} sufficiently large, and obtain the following result.
\begin{corollary}\label{cor:dim-reduction-variety}
For every $D\geq 1$ and $\epsilon>0$, there exists $c=c(D,\epsilon)>0$ such that, for every sufficiently large $K$, there exist
\[
K^{c} \leq \tilde{K}_1\leq \tilde{K}_2\leq \dotsb \leq \tilde{K}_{D} \leq \sqrt{K}
\]
such that, for every non-zero polynomial $P$ of degree $D$, there exist collections of pairwise disjoint cubes $\calG_{j} \subset \Part{1/\tilde{K}_{j}}$, $j=1,\dotsc,D$, such that
\[
N_{1/K}(Z_{P}) \cap [0,1]^{d}
\subset
\bigcup_{j=1}^{D} \bigcup_{\Box \in \calG_{j}} \Box
\]
and
\begin{equation}
\norm[\Big]{ \sum_{\Box \in \calG_{j}} f_{\Box} }_{p}
\lesssim_{D,\epsilon}
\tilde{K}_{j}^{\Lambda+\epsilon} \Big( \sum_{\Box \in \calG_{j}} \norm{f_{\Box}}^p_{p}\Big)^{1/p}.
\end{equation}
\end{corollary}
It appears somewhat unfortunate that the constant $c$ in Corollary~\ref{cor:dim-reduction-variety} depends also on $\epsilon$.
This could make quantification of the $C_{\epsilon}\delta^{-\epsilon}$ loss in Theorem~\ref{main_theorem} in the way of \cite{arxiv:1711.01202} less convenient.
However, currently the main obstacle in that direction is the unquantified transversality in Lemma~\ref{lem:transverse}.

\subsection{Transversality}\label{sec:gen:transverse}
To introduce the multilinear decoupling inequality, we first need to introduce the notion of transversality.
Let $(V_{j})_{j=1}^{M}$ be a tuple of linear subspaces $V_j \subset \R^{n+d}$ of dimension $d$.
Let $\pi_j: \R^{n+d}\to V_j$ denote the orthogonal projection onto $V_j$.
The \emph{Brascamp--Lieb constant} $BL((V_{j})_{j=1}^{M})$ is the smallest constant (possibly $\infty$) such that the inequality
\begin{equation}
\label{eq:BL-const-def}
\int_{\R^{n+d}} \prod_{j=1}^M f_j (\pi_j (x))^{\frac{n+d}{d M}} d x
\leq
BL((V_{j})_{j=1}^{M}) \prod_{j=1}^M \bigl( \int_{V_{j}} f_j (y) d y \bigr)^{\frac{n+d}{ d M}}
\end{equation}
holds for all non-negative measurable functions $f_j: V_j\to \R$.
\begin{definition}[transversality]\label{transversality}
Let $\nu>0$.
A tuple of subsets $R_1, \dotsc, R_M\subset [0, 1]^d$ is called \emph{$\nu$-transverse} if, for every choice $t_j\in R_j$, we have
\begin{equation}
BL((V(t_j))_{j=1}^M)\le \nu^{-1},
\end{equation}
where $V(t)$ denotes the tangent space of the surface $\calS_{d, n}$ at $t$.
\end{definition}

\begin{remark}
The notion of transversality in Definition~\ref{transversality} goes back to \cite{MR3548534}.
In the case $d=n$, $M=2$, it specializes to the notions used in \cite{MR3447712,MR3848437}, where transversality means that $V(t_{1}),V(t_{2})$ do not share common directions.
In the case $n=1$, $P_{1}$ positive definite, $M=d+1$, it specializes to the notion used in \cite{MR3374964,MR3592159}, because the associated Brascamp--Lieb inequality is the Loomis--Whitney inequality, and the best constant in that inequality is the reciprocal of the volume of the parallelepiped spanned by the normal directions of $V_{j}$'s.

In general, a tuple $(R_{j})_{j=1}^{M}$ can only be transverse if $M$ is sufficiently large depending on the surface $\calS_{d,n}$.
How large exactly $M$ can become depends on the choice of $K$ in the proof of Theorem~\ref{thm:dec:general}.
\end{remark}

One of the main results of Bennett, Carbery, Christ, and Tao \cite{MR2661170} says that
\begin{equation}
BL((V(t_j))_{j=1}^M)<\infty
\end{equation}
if and only if the spaces $(V(t_j))_{j=1}^M$ satisfy the condition
\begin{equation}\label{eq:BCCT-condition}
\dim(V)\le \frac{d+n}{d M}\sum_{j=1}^M \dim(\pi_{t_j}(V))
\end{equation}
for every linear subspace $V\subset \R^{d+n}$, where $\pi_{t}$ denotes the orthogonal projection onto $V(t)$.
Moreover, from \cite{MR3723636} we know that the function $(t_j)_{j=1}^M \mapsto BL((V(t_j))_{j=1}^M)$ is continuous (with values in $[0,\infty]$).
Indeed, it is even H\"older continuous \cite{arxiv:1811.11052}.

In order to ensure existence of transverse sets, we have to make some assumptions on the surface $\calS$.
\begin{hypothesis}\label{hyp:transverse}
Suppose that for every subspace $V \subset \R^{d+n}$ one of the following holds.
\begin{enumerate}
\item $\dim \pi_{t}(V) \geq \frac{d}{d+n} \dim V$ for \emph{every} $t\in \R^{d}$, or
\item $\dim \pi_{t}(V) > \frac{d}{d+n} \dim V$ for \emph{some} $t\in \R^{d}$.
\end{enumerate}
\end{hypothesis}
In the cases of Theorem~\ref{main_theorem}, Hypothesis~\ref{hyp:transverse} will be verified in Section~\ref{sec:spec:transverse}.

\begin{lemma}\label{lem:transverse}
Assuming Hypothesis~\ref{hyp:transverse}, there exists $\theta>0$ such that the following holds.
For every $K$, there exists $\nu_{K}>0$ such that, for every subcollection $\calR \subset \Part{1/K}$, one of the following alternatives holds.
\begin{enumerate}
\item $\calR$ is $\nu_{K}$-transverse, or
\item\label{lem:transverse:near-subvariety}
there exists a subvariety $Z$ of degree at most $d$ such that
\[
\abs{\Set{R \in \calR \given 2R \cap Z\neq \emptyset}} > \theta \abs{\calR}.
\]
\end{enumerate}
\end{lemma}
Analogues of Lemma~\ref{lem:transverse} were also used in \cite{MR3614930,MR3709122,MR3994585,arxiv:1811.02207}.
We note that, in Lemma~\ref{lem:transverse}, the alternative \ref{lem:transverse:near-subvariety} holds trivially if $\abs{\calR}$ is sufficiently small depending on $d,n$.
\begin{remark}
The bound $d$ on the degree is not optimal in many situations.
For instance, in the case $n=1$ considered in \cite{MR3374964,MR3736493}, we can use a subvariety of degree $1$, that is, a hyperplane.
This follows from Lemma~\ref{lem:BCCT:n=1}.

In the case $d=n$ considered in \cite{MR3447712,MR3848437}, it might have previously seemed important that only certain specific varieties can obstruct transversality.
Thanks to Corollary~\ref{cor:dim-reduction-variety}, we can afford not to keep track of which varieties may or may not arise here.
\end{remark}
\begin{proof}[Proof of Lemma~\ref{lem:transverse}]
Let $V \subset \R^{d+n}$ be a subspace.
If the first alternative in Hypothesis~\ref{hyp:transverse} holds, then the BCCT condition~\eqref{eq:BCCT-condition} holds for that subspace with any choice of $t_{j}$.

Suppose now that the second alternative in Hypothesis~\ref{hyp:transverse} holds.
The restriction of the projection operator $\pi_{t}$ to $V$ can be written in coordinates as a $d \times \dim V$ matrix. In other words, let $V$ be a linear subspace spanned by $v_1, \dots, v_{\dim V}$, let $T_j(t)$ denote the tangent vector to the surface $\mathcal{S}$ in the $t_j$ variable, then $\dim \pi_{t}(V)$ is equal to the rank of the matrix 
\[
(\langle v_i, T_j(t))_{\substack{1\le i\le \dim V; \\ 1\le j\le d}}
\]
All the entries of this matrix are linear polynomials in $t$ as our surface is quadratic.  By the hypothesis, some minor determinant of that matrix of order $>\frac{d}{d+n} \dim V$ does not vanish for some $t$.
Hence, that minor determinant is a non-trivial polynomial of degree at most $d$, and the dimension of the projection is $>\frac{d}{d+n} \dim V$ outside its zero set $Z$.

In particular, the BCCT condition~\eqref{eq:BCCT-condition} for $(t_{j})_{j=1}^{M}$ holds for $V$, provided that
\[
\dim(V) \leq \frac{d+n}{dM} \sum_{j : t_{j} \not\in Z} \lfloor \frac{d}{d+n} \dim (V) + 1 \rfloor,
\]
which can be equivalently written as
\[
\abs{\Set{j \given t_{j} \not\in Z}}/M \geq
\dim(V) \frac{d}{d+n} / \lfloor \frac{d}{d+n} \dim (V) + 1 \rfloor.
\]
The number on the right-hand side is $<1$ and can take only finitely many values, since $\dim(V)$ is a natural number $\leq d+n$.
Let $\theta$ be $1$ minus the maximum of the right-hand side over $V$.
Then the BCCT condition follows from
\[
\abs{\Set{j \given t_{j} \in Z}} \leq \theta M.
\]
This clearly holds for $t_{j}\in R_{j}$, where $(R_{j})_{j=1}^{M}$ is an enumeration of $\calR$, unless the second alternative of the Lemma holds.

Finally, if the second alternative of the lemma does not hold, then the set of tuples $(t_{j})$ with $t_{j}\in R_{j}$ is a compact subset of the set of tuples for which the BCCT condition holds.
Hence, by continuity of the Brascamp--Lieb constant, there exists a lower bound $\nu_{K}$ on the transversality of the tuple $\calR$.
\end{proof}
\begin{remark}
The use of a compactness argument makes the transversality bound $\nu_{K}$ ineffective.
\end{remark}

\subsection{Multilinear decoupling}\label{sec:gen:multilinear}
We use a version of the Bourgain--Guth scheme \cite{MR2860188} that goes back to an article of Bourgain, Demeter, and the first author \cite{MR3709122}.
In this version, the degree of multi-linearity is allowed to range in an interval depending on $K$.

For a positive integer $K$ and $0 < \delta < K^{-1}$, the \emph{multilinear decoupling constant} $\MulDec^{p}(\delta, K)$ is the smallest constant such that the inequality
\begin{multline}
\label{eq:multilin-Dec}
\Bigl( \int_{\R^{d+n}} \bigl( \avprod \norm{ f_{R_{i}} }_{\avL^{p}(B(x,K))} \bigr)^{p} \dif x \Bigr)^{1/p}\\
\le \MulDec^{p}(\delta, K)
\avprod \Bigl( \sum_{J \in \Part[R_i]{\delta}} \norm{f_J}_{L^p(\R^{d+n})}^{p} \Bigr)^{\frac{1}{p}}
\end{multline}
holds for every $\nu_{K}$-transverse tuple $R_{1},\dotsc,R_{M} \in \Part{K^{-1}}$ with $1\leq M \leq K^{d}$, where $\nu_{K}>0$ is as in Lemma~\ref{lem:transverse}.
Given $f_{J}$, $J\in\Part{\delta}$, we write here and later
\begin{equation}
\label{eq:summation-convention}
f_{\alpha} := \sum_{J \in \Part[\alpha]{\delta}} f_{J}
\end{equation}
for dyadic cubes $\alpha$ of scale $\geq \delta$.

For comparison with other literature, we note that the quantity on the left-hand side of \eqref{eq:multilin-Dec} is equivalent to
\[
\bigl( \sum_{B' \in \Cov{K}} \avprod \norm{ f_{R_{i}} }_{L^{p}(B')}^{p} \bigr)^{1/p},
\]
where $\Cov{K}$ denotes a finitely overlapping cover of $\R^{d+n}$ by balls of radius $K$.
Note the absence of average in the subscript $L^{p}(B')$.

LHS of \eqref{eq:multilin-Dec} can be thought of as morally equivalent to $\norm{ \avprod \abs{f_{R_{i}}}}_{p}$, since by the uncertainty principle the functions $f_{R_{i}}$ are morally constant at scale $K$.
Following \cite{MR3592159}, we use the formally larger averaged quantity, because it can be more easily obtained in the Bourgain--Guth argument.

As for $\Dec^{p}$, we will omit the exponent $p$ in $\MulDec^{p}$ when it is clear from context.

\subsection{Bourgain--Guth argument}\label{sec:gen:bourgain-guth}
From H\"older's inequality, it follows that
\begin{equation}\label{180713e3.4}
\MulDec^{p}(\delta, K)
\lesssim
\Dec^{p}(\delta).
\end{equation}
Here $K$ is much smaller compared with $\delta^{-1}$, and the implicit constant does not depend on $K$. The Bourgain--Guth argument shows that the converse inequality also holds, up to some lower-dimensional terms.
This is made precise in the following result.
\begin{proposition}\label{181209prop5.5}
Let $2 \leq p < \infty$.
Assume Hypothesis~\ref{hyp:transverse} and Hypothesis~\ref{hyp:lower-dim}.
Then, for each $\epsilon>0$, there exists $K$ such that
\begin{equation}
\Dec^{p}(\delta)
\lesssim_{\epsilon}
\delta^{-\Lambda-\epsilon}
+ \delta^{-\epsilon} \max_{\delta\le \delta'\le 1; \delta' \text{dyadic}} \Big[\big(\frac{\delta'}{\delta} \big)^{\Lambda} \MulDec^{p}(\delta', K)\Big].
\end{equation}
\end{proposition}

It is not difficult to see that Proposition~\ref{181209prop5.5} can be proven by iterating the following result $O(\frac{\abs{\log\delta}}{\log K})$ many times.
It is important to choose $K$ large enough depending on $\epsilon$, since we lose a constant $C_{\epsilon}$ in every step of the iteration.

\begin{proposition}
\label{prop:bourgain-guth-arg}
Let $1 \leq p < \infty$.
Assume Hypothesis~\ref{hyp:transverse} and Hypothesis~\ref{hyp:lower-dim}.
Then there exists a small constant $c(d, \epsilon)>0$ such that for every $K\ge 2$ and $0<\delta<1/K$, we have
\begin{equation}
\label{eq:BG-arg}
\Dec^{p}(\delta)
\leq
C_{\epsilon} \sup_{K^{c(d,\epsilon)} \leq \tilde{K} \leq K} \tilde{K}^{\Lambda+\epsilon} \Dec^{p}(\delta\tilde{K})
+ C_{K} \MulDec^{p}(\delta, K).
\end{equation}
\end{proposition}
\begin{proof}[Proof of Proposition~\ref{prop:bourgain-guth-arg}]
Fix $f_{J}$, $J\in\Part{\delta}$, and recall the convention \eqref{eq:summation-convention}.

Let $B' \in \Cov{K}$, and initialize
\begin{equation}
\label{eq:initial-stock}
\calS_{0}(B') := \Set{ \alpha \in \Part{K^{-1}} \given \norm{f_\alpha}_{L^{p}(B')} \ge K^{-d} \max_{\alpha' \in \Part{K^{-1}}} \norm{f_{\alpha'}}_{L^{p}(B')} }.
\end{equation}
We repeat the following algorithm.

Let $m\geq 0$.
If $\calS_{m}(B') = \emptyset$ or $\calS_{m}(B')$ is $\nu_{K}$-transverse, then we set
\[
\calT(B') := \calS_{m}(B').
\]
Otherwise, by Lemma~\ref{lem:transverse}, there exists a subvariety $Z$ of degree $d$ such that
\begin{equation}
\label{eq:stock-near-variety}
\abs{\Set{ \alpha\in \calS_{m}(B') \given 2\alpha \cap Z \neq \emptyset}}
\geq
\theta \abs{\calS_{m}(B')}.
\end{equation}
Let $\calG_{m,j}(B') := \calG_{j}$ be given by Corollary~\ref{cor:dim-reduction-variety}.

Repeat the algorithm with
\begin{align*}
\calS_{m+1}(B')
&:=
\calS_{m}(B') \setminus \bigcup_{j=1}^{D} \bigcup_{\Box \in \calG_{m,j}(B')} \Part{\Box,1/K}.
\end{align*}
Since in each step we remove at least a fixed proportion $\theta$ of $\calS_{m}(B')$, this algorithm terminates after $O(\log K)$ steps.

To avoid multiple counting, we introduce
\begin{equation}\label{new_boxes}
\widetilde{\calG}_{m,j}(B') := \Bigl( \calG_{m,j}(B') \setminus \bigcup_{0\le m' < m} \calG_{m',j}(B') \Bigr) \setminus \bigcup_{1\le j' < j} \bigcup_{m'} \bigcup_{\Box\in \calG_{m', j'}(B')} \Part{\Box, 1/\tilde{K_j}}.
\end{equation}
We estimate
\begin{align}
\norm{ f }_{L^{p}(B')}
\label{eq:BG':small}&\leq
\sum_{\alpha \in \Part{K^{-1}} \setminus \calS_{0}(B')} \norm{ f_{\alpha} }_{L^{p}(B')}\\
\label{eq:BG':variety}&+
\sum_{m \lesssim \log K} \sum_{j=1}^{D} \norm{ \sum_{\beta \in \widetilde{\calG}_{m,j}(B')} f_{\beta} }_{L^{p}(B')}\\
\label{eq:BG':transverse}&+
\sum_{\alpha \in \calT(B')} \norm{ \ED_{\alpha} g }_{L^{p}(B')}
\end{align}

By definition of $\calS_{0}(B')$, we obtain
\[
\eqref{eq:BG':small}
\lesssim
\max_{\alpha' \in \Part{1/K}}
\norm{ f_{\alpha'} }_{L^{p}(B')}.
\]
By Corollary~\ref{cor:dim-reduction-variety} and a simple localization argument as in Remark~\ref{rem:global-to-local}, we have
\begin{multline*}
\eqref{eq:BG':variety}
\lesssim_{\epsilon}
\sum_{m\lesssim \log K} \sum_{j=1}^{D} \tilde{K}_{j}^{\Lambda+\epsilon} \Bigl( \sum_{\beta \in \widetilde{\calG}_{m,j}(B')} \norm{ f_{\beta} }_{L^{p}(w_{B'})}^{p} \Bigr)^{1/p}\\
\lesssim
(\log K) \sum_{j=1}^{D} \tilde{K}_{j}^{\Lambda+\epsilon} \Bigl( \sum_{\beta \in \Part{1/\tilde{K}_{j}}} \norm{ f_{\beta} }_{L^{p}(w_{B'})}^{p} \Bigr)^{1/p}
\end{multline*}
If $\calT(B') \neq \emptyset$, then by definition of $\calS_{0}(B')$ we obtain
\[
\eqref{eq:BG':transverse}
\lesssim
K^{C}
\min_{\alpha' \in \calT(B')} \norm{ f_{\alpha'} }_{L^{p}(B')}
\leq
K^{C} \max_{1 \leq M \leq K^{d}} \max_{\substack{\alpha_{1},\dotsc,\alpha_{M} \in \Part{K^{-1}}\\ \nu_{K}-\text{transverse}}} \avprod \norm{ f_{\alpha_{i}} }_{L^{p}(B')}.
\]
Next, we sum over all balls $B'\subset \R^{d+n}$ and obtain
\begin{align}
\notag
\norm{f}_{L^{p}(\R^{d+n})}
&\leq
\bigl( \sum_{B' \in \Cov{K}} \norm{f}_{L^{p}(B')}^{p} \bigr)^{1/p}\\
\label{eq:BG:small} & \lesssim
\bigl( \sum_{B' \in \Cov{K}} \max_{\alpha \in \Part{K^{-1}}} \norm{f_{\alpha}}_{L^{p}(B')}^{p} \bigr)^{1/p}\\
\label{eq:BG:variety}&+
(\log K) C_{\epsilon} \sum_{j=1}^{D} \tilde{K}_{j}^{\Lambda+\epsilon}
\Bigl( \sum_{\beta \in \Part{1/\tilde{K}_{j}}} \norm{ f_{\beta} }^p_{L^{p}(\R^{d+n})}\Bigr)^{1/p}\\
\label{eq:BG:transverse}&+
K^{C} \bigl( \sum_{B' \in \Cov{K}} \max_{1 \leq M \leq K^{d}} \max_{\substack{\alpha_{1},\dotsc,\alpha_{M} \in \Part{K^{-1}}\\ \nu_{K}-\text{transverse}}} \avprod \norm{ f_{\alpha_{i}} }_{L^{p}(B')}^{p} \bigr)^{1/p}
\end{align}
Let us pause and remark that it is in this step that we require $\log \tilde{K_j}\approx_{d, n, \epsilon} \log K$.
We will absorb the factor $\log K$ by $\tilde{K_j}^{\epsilon}$.

In the term \eqref{eq:BG:small}, we bound $\max_{\alpha}$ by $\ell^p_{\alpha}$ and obtain 
\[
\bigl(\sum_{\alpha \in \Part{K^{-1}}} \norm{f_{\alpha}}_{L^{p}(\R^{d+n})}^{p} \bigr)^{1/p}.
\]
Each term $\norm{f_{\alpha}}_{L^{p}(\R^{d+n})}$ is a rescaled version of the left hand side of the above inequality.
Therefore, one can use the definition of the decoupling constant and scaling.
The same argument is also applied to \eqref{eq:BG:variety}.

In the last term \eqref{eq:BG:transverse}, by definition of the multilinear decoupling constant \eqref{eq:multilin-Dec}, we estimate
\begin{multline*}
\eqref{eq:BG:transverse}
\lesssim_{K}
\bigl( \sum_{1 \leq M \leq K^{d}} \sum_{\substack{\alpha_{1},\dotsc,\alpha_{M} \in \Part{K^{-1}}\\ \nu_{K}-\text{transverse}}} \sum_{B' \in \Cov{K}} \avprod \norm{ f_{\alpha_{i}} }_{L^{p}(B')}^{p} \bigr)^{1/p}\\
\leq
\MulDec(\delta, K)
\bigl( \sum_{1 \leq M \leq K^{d}} \sum_{\substack{\alpha_{1},\dotsc,\alpha_{M} \in \Part{K^{-1}}\\ \nu_{K}-\text{transverse}}} \avprod \bigl( \sum_{J \in \Part[\alpha_{i}]{\delta}} \norm{ f_{J} }_{L^{p}(\R^{d+n})}^{p} \bigr) \bigr)^{1/p}\\
\leq
\MulDec(\delta, K)
\bigl( \sum_{1 \leq M \leq K^{d}} \prod_{i=1}^{M} \sum_{\alpha_{i} \in \Part{K^{-1}}} \bigl( \sum_{J \in \Part[\alpha_{i}]{\delta}} \norm{ f_{J} }_{L^{p}(\R^{d+n})}^{p} \bigr)^{\frac{1}{M}} \bigr)^{1/p}\\
\lesssim_{K}
\MulDec(\delta, K)
\bigl( \sum_{1 \leq M \leq K^{d}} \prod_{i=1}^{M} \bigl( \sum_{\alpha_{i} \in \Part{K^{-1}}} \sum_{J \in \Part[\alpha_{i}]{\delta}} \norm{ f_{J} }_{L^{p}(\R^{d+n})}^{p} \bigr)^{\frac{p}{p M}} \bigr)^{1/p}\\
\lesssim_{K}
\MulDec(\delta, K)
\bigl( \sum_{J \in \Part{\delta}} \norm{ f_{J} }_{L^{p}(\R^{d+n})}^{p} \bigr)^{1/p}.
\end{multline*}
Since $f_{J}$ were arbitrary, this concludes the proof.
\end{proof}

\subsection{Ball inflation}\label{sec:ball-inflation}
The following estimate, which relies on Kakeya--Brascamp--Lieb type inequalities, was introduced in \cite{MR3548534}.
We refer to \cite[Lemma 3.1]{arxiv:1811.02207} for a simplified proof.
\begin{proposition}\label{prop:ball-inflation}
Let $K\geq 1$ be a dyadic integer and $0 < \delta \leq \rho \leq 1/K$.
Let $\Set{R_j}_{j=1}^M \subset \Part{1/K}$ be a $\nu$-transverse collection of cubes.
Let $B\subset \R^{d+n}$ be a ball of radius $\rho^{-2}$.
Then, for each $1 \leq t < \infty$, we have
\begin{equation}\label{0727e4.14ha}
\begin{split}
& \avL^{\frac{d+n}{d} t}_{x \in B} \avprod \ell^{t}_{J_i \in \Part[R_i]{\rho}} \norm{f_{J_i}}_{\avL^{t}(w_{B(x,1/\rho)})}\\
& \lesssim \nu^{-\frac{d}{t(d+n)}}
\avprod \ell^{t}_{J_i \in \Part[R_i]{\rho}} \norm{f_{J_i}}_{\avL^{t}(w_B)}
\end{split}
\end{equation}
Here
\[
\avL^{p}_{x \in B}(\cdot):= \bigl( \frac{1}{\abs{B}} \int_B \abs{\cdot}^p \bigl)^{1/p},
\]
and for a countable set $\mathcal{P}$ and a sequence $\{a_J\}_{J\in \mathcal{P}}$, 
\[
\ell^t_{J\in \mathcal{P}}a_J:=\big(\sum_{J\in \mathcal{P}}|a_J|^t\big)^{1/t}.
\]
\end{proposition}

\subsection{Bourgain--Demeter iteration}\label{sec:bourgain-demeter}
In this section, we present a version of the iteration argument of Bourgain and Demeter.
Its $\ell^{2}L^{p}$ version was introduced in \cite{MR3374964}, and the $\ell^{p}L^{p}$ version in \cite{MR3736493}.
The simplified version below is a special case of the iteration in \cite{arxiv:1811.02207}.

Throughout this section, let $R_1,\dotsc,R_M \in \Part{1/K}$ be $\nu_{K}$-transverse cubes.

For $\rho \in 2^{-\N}$, we define the quantity
\begin{align*}
A_p(\rho)
&:=
L^{p}_{x} \avprod \ell^{2}_{Q\in \Part[R_i]{\rho}} \norm{f_{Q}}_{\avL^{2}(w_{B(x,1/\rho)})}.
\end{align*}
Here $L^p_x$ refers to taking the $L^p$ norm of a function depending on the $x$ variable.
We caution the reader that the quantities denoted by $A$ in \cite{MR3592159} would correspond to our $A$ with $L^{p}_{x}$ replaced by $\avL^{p}_{x\in B}$ for a large ball $B$.

Let $\tilde p:=\max(2,pd/(d+n))$, and define $\kappa = \kappa(p) \in [0,1]$ by
\[
\frac{1}{\tilde{p}}=\frac{1-\kappa}{2}+\frac{\kappa}{p}.
\]
It will be important that $\kappa \leq 1/2$ if and only if $p \leq \frac{2(d+2n)}{d}$.
\begin{proposition}
\label{prop:iter}
For each $2 \leq p < \infty$, we have
\begin{equation}
\label{eq:33}
A_p(\rho)
\lesssim
\nu^{-1/\tilde{p}} \rho^{-d(1/2-1/\tilde{p})} A_{p}(\rho^{2})^{1-\kappa}
\Bigl( \Dec^{p}(\delta/\rho) \avprod \ell^{p}_{J \in \Part[R_i]{\delta}} \norm{ f_{J} }_{p} \Bigr)^{\kappa}
\end{equation}
\end{proposition}
\begin{proof}
Using ball inflation from scale $\rho$ to scale $\rho^{2}$, we obtain
\begin{align*}
A_p(\rho)
&=
L^{p}_{x} \avL^{p}_{\tilde{x} \in B(x,1/\rho^{2})} \avprod \ell^{2}_{Q\in \Part[R_i]{\rho}} \norm{f_{Q}}_{\avL^{2}(w_{B(\tilde{x},1/\rho)})}\\
\text{by H\"older's inequality } &\lesssim
\rho^{-d(1/2-1/\tilde{p})} L^{p}_{x} \avL^{\frac{\tilde{p} (d+n)}{d}}_{\tilde{x} \in B(x,1/\rho^{2})} \avprod \ell^{\tilde{p}}_{Q\in \Part[R_i]{\rho}} \norm{f_{Q}}_{\avL^{\tilde p}(w_{B(\tilde{x},1/\rho)})}\\
\text{by Prop.~\ref{prop:ball-inflation} } &\lesssim
\nu^{-1/\tilde{p}} \rho^{-d(1/2-1/\tilde{p})} L^{p}_{x} \avprod \ell^{\tilde{p}}_{Q\in \Part[R_i]{\rho}} \norm{f_{Q}}_{\avL^{\tilde p}(w_{B(x,1/\rho^{2})})}\\
\text{by H\"older's inequality } &\leq
\nu^{-1/\tilde{p}} \rho^{-d(1/2-1/\tilde{p})} \Big( L^{p}_{x} \avprod \ell^{2}_{Q\in \Part[R_i]{\rho}} \norm{f_{Q}}_{\avL^{2}(w_{B(x,1/\rho^{2})})} \Big)^{1-\kappa}\\
&\quad\cdot
\Big( L^{p}_{x} \avprod \ell^{p}_{Q\in \Part[R_i]{\rho}} \norm{f_{Q}}_{\avL^{p}(w_{B(x,1/\rho^{2})})} \Big)^{\kappa}
\end{align*}
By $L^{2}$ orthogonality, the first bracket is
\[
\lesssim
L^{p}_{x} \avprod \ell^{2}_{Q\in \Part[R_i]{\rho^{2}}} \norm{f_{Q}}_{\avL^{2}(w_{B(x,1/\rho^{2})})}
=
A_{p}(\rho^{2}).
\]
In the second bracket, we estimate
\begin{equation}
\label{eq:37}
\begin{split}
&
L^{p}_{x} \avprod \ell^{p}_{Q\in \Part[R_i]{\rho}} \norm{f_{Q}}_{\avL^{p}(w_{B(x,1/\rho^2)})}\\
\text{by H\"older's inequality }&\leq
\avprod L^{p}_{x} \ell^{p}_{Q\in \Part[R_i]{\rho}} \norm{f_{Q}}_{\avL^{p}(w_{B(x,1/\rho^2)})}\\
\text{by Minkowski's inequality }&\leq
\avprod \ell^{p}_{Q\in \Part[R_i]{\rho}} L^{p}_{x} \norm{f_{Q}}_{\avL^{p}(w_{B(x,1/\rho^2)})}\\
&\lesssim
\avprod \ell^{p}_{Q\in \Part[R_i]{\rho}} \norm{f_{Q}}_{p}\\
\text{by scaling }&\lesssim
\avprod \ell^{p}_{Q\in \Part[R_i]{\rho}} \bigl( \Dec(\delta/\rho) \ell^{p}_{J\in \Part[Q]{\delta}} \norm{f_{J}}_{p} \bigr)\\
&=
\Dec(\delta/\rho) \avprod \ell^{p}_{J\in \Part[R_i]{\delta}} \norm{f_{J}}_{p}.
\qedhere
\end{split}
\end{equation}
\end{proof}

\begin{proposition}\label{prop:bourgain-demeter}
Let $2 \leq p \leq 2 + \frac{4n}{d}$, and suppose that
\begin{equation}
\label{eq:40}
\Dec^{p}(\delta) \lesssim \delta^{-\eta}
\end{equation}
for some $\eta = d(1/2-1/p) + \sigma$ with $\sigma>0$.
Then, for every $K$, we have
\begin{equation}
\label{eq:2}
\MulDec^{p}(\delta,K) \lesssim_{K} \delta^{-\eta+\tilde{\eta}(\sigma)},
\end{equation}
where $\tilde{\eta} : (0,\infty) \to (0,\infty)$ is a monotonically increasing function depending on $p$.
\end{proposition}
\begin{proof}
Choose $\nu_{K}$-transverse $R_{1},\dotsc,R_{M} \in \Part{1/K}$.
Choose functions $f_{J}$ with
\[
\ell^{p}_{J \in \Part[R_i]{\delta}} \norm{f_J}_{p} = 1.
\]
Let $m\in\N$ be chosen later.
It suffices to consider $\delta$ that are powers of $2^{2^{m}}$.
Let $\rho=\delta^{2^{-m}}$.
Then
\begin{equation}
\label{eq:1}
\begin{split}
\MoveEqLeft\relax
L^{p}_{x} \avprod \norm{f_{R_i}}_{\avL^p(B(x,K))}\\
&=
L^{p}_{x} \avL^{p}_{\tilde{x} \in B(x,1/\rho)} \avprod \norm{f_{R_i}}_{\avL^p(B(\tilde{x},K))}\\
\text{by H\"older's inequality }&\leq
L^{p}_{x} \avprod \avL^{p}_{\tilde{x} \in B(x,1/\rho)} \norm{f_{R_i}}_{\avL^p(B(\tilde{x},K))}\\
&\lesssim
L^{p}_{x} \avprod \norm{f_{R_i}}_{\avL^p(B(x,1/\rho))}\\
\text{by Minkowski's inequality}&\leq
L^{p}_{x} \avprod \ell^{1}_{J\in\Part[R_i]{\rho}} \norm{f_{J}}_{\avL^p(B(x,1/\rho))} \\
\text{by H\"older's inequality }&\leq
\rho^{-d/2}
L^{p}_{x} \avprod \ell^{2}_{J\in\Part[R_i]{\rho}} \norm{f_{J}}_{\avL^p(B(x,1/\rho))}\\
\text{by Bernstein's inequality}&\lesssim
\rho^{-d/2}
L^{p}_{x} \avprod \ell^{2}_{J\in\Part[R_i]{\rho}} \norm{f_{J}}_{\avL^2(w_{B(x,1/\rho)})}\\
&=
\rho^{-d/2}
A_p(\rho)
\end{split}
\end{equation}

Iterating Proposition~\ref{prop:iter}, starting with $\rho=\delta^{2^{-m}}$, until we get to $\rho=\delta$, at which point we use H\"older's inequality, we get
\begin{align*}
A_{p}(\rho)
&\lesssim
\prod_{l=0}^{m-1} \Bigl( C\nu^{-1/\tilde{p}} \rho^{- 2^{l} d (1/2-1/\tilde{p})} \Dec(\delta/\rho^{2^{l}})^{\kappa} \Bigr)^{(1-\kappa)^{l}}
\cdot
(\delta^{-d(1/2-1/p)})^{(1-\kappa)^{m}}.
\end{align*}
By the assumption on the linear decoupling constant, this is
\begin{equation}
\label{eq:45}
\begin{split}
&\lesssim_{m,\nu}
\prod_{l=0}^{m-1} \Bigl( \rho^{- 2^{l} d (1/2-1/\tilde{p})} \delta^{-\eta \kappa}/\rho^{-2^{l} \eta \kappa} \Bigr)^{(1-\kappa)^{l}}
\cdot \delta^{-d(1/2-1/p)(1-\kappa)^{m}}
\\ &=
\delta^{-\eta \kappa \sum_{l=0}^{m-1} (1-\kappa)^{l}} \rho^{(\eta \kappa - d(1/2-1/\tilde{p})) \cdot \sum_{l=0}^{m-1} 2^{l} (1-\kappa)^{l}} \delta^{-d(1/2-1/p)(1-\kappa)^{m}}
\\ &=
\delta^{-\eta (1- (1-\kappa)^{m})} \rho^{(\kappa \sigma) \cdot \sum_{l=0}^{m-1} 2^{l} (1-\kappa)^{l}} \delta^{-(\eta-\sigma)(1-\kappa)^{m}},
\end{split}
\end{equation}
where we used that $d(1/2-1/\tilde{p}) = \kappa d(1/2-1/p) = \kappa(\eta-\sigma)$.

Recalling the factor $\rho^{-d/2}$ from \eqref{eq:1} and taking a supremum over all $R_i$ and $f_{J}$ as above, we deduce
\begin{equation}
\label{eq:49}
\MulDec(\delta, K)
\lesssim_{m, \nu}
\delta^{-\eta + \sigma(1-\kappa)^{m}} \rho^{-d/2+(\kappa\sigma) \cdot \sum_{l=0}^{m-1} 2^{l} (1-\kappa)^{l}}.
\end{equation}

By the hypothesis on $p$, we have $\kappa \leq 1/2$.
Hence,
\[
\eqref{eq:49}
\leq
\delta^{-\eta + \sigma(1-\kappa)^{m}} \rho^{-d/2+(\kappa\sigma) \cdot m}.
\]
Choosing $m = \lceil d/(2\kappa\sigma) \rceil$, we obtain the claim \eqref{eq:2} with $\tilde{\eta}(\sigma) = \sigma (1-\kappa)^{\lceil d/(2\kappa\sigma) \rceil}$.
\end{proof}

\begin{theorem}\label{thm:dec:general}
Let $2 \leq p \leq 2 + \frac{4n}{d}$.
Assume Hypothesis~\ref{hyp:transverse} and Hypothesis~\ref{hyp:lower-dim} with $\Lambda \geq d(1/2-1/p)$.
Then, for every $\epsilon>0$, we have
\[
\Dec^{p}(\delta) \lesssim_{\epsilon} \delta^{-\Lambda-\epsilon}.
\]
\end{theorem}
\begin{proof}
It is easy to see that $\Dec(\delta) \lesssim \delta^{-\eta}$ for some $\eta = d(1/2-1/p) + \sigma$ with $\sigma>0$.
If $\eta \leq \Lambda$, then we are done.
Otherwise, we will be able to decrease $\eta$.
Substituting the conclusion of Proposition~\ref{prop:bourgain-demeter} into the conclusion of Proposition~\ref{prop:bourgain-guth-arg} gives
\begin{align*}
\Dec(\delta)
&\leq
C_{\epsilon} \sup_{\log \tilde{K} \approx \log K}
\tilde{K}^{\Lambda+\epsilon} \Dec(\delta \tilde{K})
+
C_{K} \MulDec(\delta, K)\\
&\leq
C_{\epsilon} \sup_{\log \tilde{K} \approx \log K}
\tilde{K}^{\Lambda+\epsilon} (\delta \tilde{K})^{-\eta}
+
C_{K} \delta^{-\eta'}
\end{align*}
with $\eta' = \eta - \tilde{\eta}(\sigma)$ for any $K$.
Iterating this inequality $O(\frac{\log \delta^{-1}}{\log K})$ times, we obtain
\[
\Dec(\delta)
\lesssim
C_{K} C_{\epsilon}^{C \abs{\log\delta}/\log K } \delta^{-\max(\Lambda+\epsilon,\eta')}.
\]
Choosing $K$ large enough in terms of $C_{\epsilon}$, this gives
\[
\Dec(\delta)
\lesssim_{\epsilon}
\delta^{-\max(\Lambda,\eta')-2\epsilon}.
\]
Thus we have succeeded in decreasing $\eta$.
Iterating this, we can make $\eta$ arbitrarily close to $\Lambda$.
\end{proof}

\section{Specific surfaces}\label{sec:spec}

\subsection{Lower dimensional decoupling}\label{sec:spec:lower-dim}
In this section, we verify Hypothesis~\ref{hyp:lower-dim}.
\begin{lemma}\label{181212lem5.7}
Let $d$ and $n$ be as in \eqref{d_n_constraint} and assume \eqref{low_dim_assumption}.
Then, for every $2 \leq p \leq 2+\frac{4n}{d}$ and every $\epsilon>0$, we have
\begin{equation}\label{181212e5.34}
\Dec_{\calH}^{p}(\delta)
\lesssim_{\epsilon}
\big(\frac{1}{\delta} \big)^{d(\frac 1 2-\frac 1 p)+\epsilon}
\end{equation}
for every hyperplane $\calH$ given by \eqref{181212e4.2} with $\abs{\nabla \calL}\lesssim 1$.
\end{lemma}

\begin{proof}[Proof of Lemma~\ref{181212lem5.7}.]
The case $d=1$ is trivial, so we assume $d\geq 2$.

By a compactness argument, the implicit constant can be made uniform in $\calH$, so we concentrate on showing \eqref{181212e5.34} for any fixed $\calH$.
It suffices to find a subspace $\calH' \subset \calH$ of dimension $d-2$ on which the decoupling exponent, that is, the power of $\delta^{-1}$ in $\Dec^{p}(\delta)$  of \eqref{eq:dec-const}, is $(d-2)(1/2-1/p)$; then we can apply flat decoupling, see e.g.\ \cite[Appendix B]{arxiv:1811.02207}, in the remaining direction.

By the hypothesis \eqref{low_dim_assumption}, there exist $\calH' \subset \calH$ of dimension $d-2$ and $\lambda$ such that $\lambda_{1}P_{1}+\dotsb+\lambda_{n}P_{n}$ has full rank on $\calH'$.
By a change of variables, we may assume $\calH'=\Set{(t_{1},\dotsc,t_{d-2},0,0)}$ and $\lambda = (1,0,\dotsc,0)$.
But in this case, the claim is given by Theorem~\ref{main_theorem} for $n=1$ and $d$ replaced by $d-2$ in view of \eqref{eq:codim1-hypotheses}.
The only thing to be verified is the restriction on the exponents
\[
2+\frac{4n}{d} \leq 2+\frac{4}{d-2},
\]
which is equivalent to $n(d-2) \leq d$, and is satisfied in the cases listed in \eqref{d_n_constraint}. Also, as mentioned in the introduction, this is how the constraint on $n$ and $d$ shows up. 
\end{proof}

\subsection{Transversality}\label{sec:spec:transverse}
In this section, we verify Hypothesis~\ref{hyp:transverse}.
The case $n=1$ will be verified in Lemma~\ref{lem:BCCT:n=1}, the case $n=2, d\ge 3$ in Lemma~\ref{180717lem5.2}, and the remaining cases in Lemma~\ref{181203lemma6.4}.

We write $\R^{d+n}=\R^{d}\oplus \R^{n}=S_1\oplus S_2$.
\begin{lemma}\label{lem:proj-dim-lower-bd:simple}
Let $V \subset \R^{d+n}$ be a subspace.
Then
\[
\dim( \pi_{t}(V) ) \geq \dim (V \cap S_{1})
\]
for all $t$.
\end{lemma}
The proof of this lemma is straightforward, and we leave it out. 

\begin{lemma}\label{lem:proj-dim-lower-bd}
Suppose that \eqref{bl_assumption} holds.
Let $V \subset \R^{d+n}$ be a subspace, and let
\[
0 \leq H_{1} \leq \min(\dim(V\cap S_{1}),d-n),
\quad
0 \leq H_{2} \leq \dim(V/S_{1}).
\]
Then
\[
\dim( \pi_{t}(V) ) \geq H_{1}+H_{2}
\]
for all $t$ outside the zero set of some non-trivial polynomial of degree $H_{2}$.
\end{lemma}
\begin{remark}
Since $\R^{d+n} = S_{1} \oplus S_{2}$, we have
\[
\dim(V) = \dim(V \cap S_{1}) + \dim(V/S_{1}).
\]
\end{remark}
\begin{proof}[Proof of Lemma \ref{lem:proj-dim-lower-bd}.]
By the hypotheses on $V$, we can choose a linearly independent set $(v_{i})_{1\leq i \leq H_{1}+H_{2}} \subset V$ with $v_{i}=(w_{i},z_{i})\in \R^{d}\times\R^{n}$ such that $z_{1}=\dotsb=z_{H_{1}}=0$ and $z_{H_{1}+1},\dotsc,z_{H_{1}+H_{2}}$ are linearly independent.

Consider the vectors $n_j(t) := (e_{j}, \nabla P(t)\cdot e_j) \in \R^{d} \times \R^{n}$, $j=1,\dotsc,d$, which form a basis of the tangent space of the surface $\calS_{d, n}$ at the point $t\in \R^d$.
Then
\begin{equation}\label{181212e5.9}
\dim(\pi_t(V))
=
\rank \big(\langle v_i, n_j(t) \rangle \big)_{\substack{1\le i\le H_{1}+H_{2}\\1\le j\le d}}.
\end{equation}
The matrix on the right-hand side of \eqref{181212e5.9} can be written as
\begin{multline*}
\big(\langle v_i, n_j(t) \rangle \big)_{\substack{1\le i\le H_{1}+H_{2}\\1\le j\le d}}
=
\big(\langle w_i, e_j \rangle + \langle z_{i}, \nabla P(t)\cdot e_j \rangle  \big)_{\substack{1\le i\le H_{1}+H_{2}\\1\le j\le d}}\\
=
\big(w_i + \nabla P(t) \cdot z_{i}\big)_{1\leq i \leq H_{1}+H_{2}}.
\end{multline*}
Here, $w_i, e_j, z_i$ are all treated as column vectors.
Denote $H=H_1+H_2$.
Since $\nabla P(t)$ is linear in $t$, each $H\times H$ minor determinant of this matrix is a polynomial of degree at most $H_{2}$ in $t$.
Suppose for a contradiction that these minor determinants vanish identically.
Then also their degree $H_{2}$ homogeneous parts vanish identically, and they coincide with the corresponding $H\times H$ minor determinants of the matrix
\[
\big(w_{1}, \dotsc,w_{H_{1}}, \nabla P(t)\cdot z_{H_{1}+1}, \dotsc, \nabla P(t)\cdot z_{H_{1}+H_{2}} \big)
\]
Therefore, the latter matrix does not have full rank for any $t$.

Let $\tilde{w}_{1},\dotsc,\tilde{w}_{d-n} \in \R^{d}$ be linearly independent vectors with $\tilde{w}_{j} = w_{j}$ for $j\leq H_{1}$ and $\tilde{z}_{1},\dotsc,\tilde{z}_{n} \in \R^{n}$ be linearly independent vectors with $\tilde{z}_{j}=z_{H_{1}+j}$ for $j \leq H_{2}$.
Then the matrix
\[
\big( \tilde{w}_{1},\dotsc,\tilde{w}_{d-n}, \nabla P(t)\cdot \tilde{z}_{1}, \dotsc, \nabla P(t)\cdot \tilde{z}_{n} \big)
\]
does not have full rank for any $t \in \R^{d}$.
But the latter matrix can be factored as
\[
\begin{pmatrix}
\tilde{w}_{1},\dotsc,\tilde{w}_{d-n}, \nabla P(t)
\end{pmatrix}
\begin{pmatrix}
I_{d-n} & 0\\
0 & \tilde{z}_{1}, \dotsc, \tilde{z}_{n}
\end{pmatrix}.
\]
The latter matrix is invertible, and the former is invertible for all $t$ outside a proper subvariety by the hypothesis \eqref{bl_assumption}.
This contradiction finishes the proof.
\end{proof}

\begin{lemma}\label{lem:BCCT:n=1}
Let $d\ge 1$ and $n=1$.
For every proper linear subspace $V\subset \R^{d+n}$, it holds that
\begin{equation}
\Set{t \given \dim(\pi_t(V))< \dim(V)}
\end{equation}
is contained in a subvariety of degree $1$.
\end{lemma}

\begin{proof}[Proof of Lemma~\ref{lem:BCCT:n=1}.]
We may assume $\dim(V)=d$.
The same argument works for all other cases. 
Let $H_{2} := \dim(V/S_{1})$.
If $H_{2} = 0$, then by Lemma~\ref{lem:proj-dim-lower-bd:simple} we have
\[
\dim(\pi_{t}(V)) \geq \dim(V \cap S_{1}) = \dim(V)
\]
for all $t$.
If $H_{2}=1$, then, by Lemma~\ref{lem:proj-dim-lower-bd} with $H_{1}=\dim(V)-1$, we obtain
\[
\dim(\pi_{t}(V)) \geq \dim(V)
\]
for all $t$ outside a subvariety of degree $1$.
\end{proof}

\begin{lemma}\label{180717lem5.2}
Let $n=2$ and $d\ge 3$.
Let $V\subset \R^{d+n}$ be a non-trivial proper linear subspace.
\begin{enumerate}
\item If $1\le \dim(V)\le d-1$, then the set
\begin{equation}
\Set{t \given \dim(\pi_t(V))< \dim(V)}
\end{equation}
is contained in a subvariety of degree $2$.
\item If $d\le \dim(V)\le d+1$, then the set
\begin{equation}
\Set{t \given \dim(\pi_t(V))< \dim(V)-1}
\end{equation}
is contained in a subvariety of degree $2$.
\end{enumerate}
\end{lemma}

\begin{proof}[Proof of Lemma~\ref{180717lem5.2}.]
Let $H_{2} := \dim(V/S_{1}) \leq 2$.
If $H_{2}=0$, then by Lemma~\ref{lem:proj-dim-lower-bd:simple} we have $\dim(\pi_{t}(V))=\dim V$ for all $t$.
Otherwise, by Lemma~\ref{lem:proj-dim-lower-bd} with $H_{1}=\min(\dim(V)-H_{2},d-2)$, we obtain
\[
\dim(\pi_{t}(V)) \geq \min(\dim(V),d-2+H_{2})
\]
for all $t$ outside some subvariety of degree $\leq 2$.
This gives the claim unless $H_{2}=1$, $\dim(V)=d+1$.
But in this case $S_{1} \subset V$, so $\dim(\pi_{t}(V)) \geq d$ for all $t$ by Lemma~\ref{lem:proj-dim-lower-bd:simple}.
\end{proof}

\begin{lemma}\label{181203lemma6.4}
Let $d=n\ge 2$ and $V \subset \R^{d+n}$.
\begin{enumerate}
\item If $\dim(V)$ is odd, then
\begin{equation}
\Set{t \given \dim(\pi_t(V)) < (\dim(V)+1)/2}
\end{equation}
is contained in a subvariety of degree at most $d$.
\item If $\dim(V)$ is even, then either
\begin{equation}
\Set{t \given \dim(\pi_t(V)) < \dim(V)/2 + 1}
\end{equation}
is contained in a subvariety of degree at most $d$, or
\[
\dim(\pi_{t}(V)) \geq \dim(V)/2
\]
for all $t$.
\end{enumerate}
\end{lemma}
\begin{proof}[Proof of Lemma~\ref{181203lemma6.4}.]
Let $H_{2}:=\dim(V/S_{1})$.
If $H_{2} > \dim(V)/2$, then $H_{2}\geq (\dim(V)+1)/2$ for $\dim(V)$ odd, and $H_{2}\geq \dim(V)/2+1$ for $\dim(V)$ even.
By Lemma~\ref{lem:proj-dim-lower-bd} with $H_{1}=0$, we obtain the claim in this case.

If $H_{2} \leq \dim(V)/2$, then $\dim(V\cap S_{1}) \geq \dim(V)/2$, so, by Lemma~\ref{lem:proj-dim-lower-bd:simple}, we obtain $\dim(\pi_{t}(V)) \geq \dim(V)/2$ for all $t$.
A case distinction between $\dim(V)$ odd and even finishes the proof.
\end{proof}

\appendix
\section{Simultaneously diagonalizable forms}\label{sec:simul-diag}
Here we prove Lemma~\ref{lem:simul-diag}.
The case $d=2$ is trivial.
The case $d=3$ is contained in Demeter, Guo, and Shi \cite[Corollary 1.2]{MR3945730}.
Therefore, in this section we work with the case $d=4$.

Let us first verify the condition \eqref{bl_assumption}.
Take two linearly independent vectors $\vec{u}, \vec{v}\in \R^4$ with $\vec{u}=(u_1, \dots, u_4)$ and $\vec{v}=(v_1, \dots, v_4)$.
We need to show that
\begin{equation}
\det\begin{bmatrix}
a_1 t_1 & a_2 t_2 & a_3 t_3 & a_4 t_4\\
b_1 t_1 & b_2 t_2 & b_3 t_3 & b_4 t_4\\
u_1 & u_2 & u_3 & u_4\\
v_1 & v_2 & v_3 & v_4
\end{bmatrix}
\end{equation}
does not vanish constantly, when viewed as a polynomial in $t$.
We argue by contradiction and assume that this determinant vanishes constantly.
Then it is not difficult to see, via calculating this determinant directly, that
\begin{equation}
\det \begin{bmatrix}
u_i & u_j\\
v_i & v_j
\end{bmatrix}
=0 \text{ for every } i<j.
\end{equation}
This further implies that $u$ and $v$ are linearly dependent, which is a contradiction.

Next we verify the condition \eqref{low_dim_assumption}.
We argue by contradiction and assume that there exists a hyperplane $H\subset \R^d$ such that
\begin{equation}\label{181208e2.2}
\max\Set{\rank (P|_H), \rank (Q|_H)}\le 1.
\end{equation}
Define two diagonal matrices
\begin{equation}
M_1:=\diag(a_1, a_2, a_3, a_4) \text{ and } M_2:=\diag(b_1, b_2, b_3, b_4).
\end{equation}
The assumption \eqref{181208e2.2} implies that
\begin{equation}\label{181208e2.4}
\max\Set{ \rank (LM_1 L^T), \rank (LM_2 L^T)}\le 1,
\end{equation}
for some $3\times 4$ matrix $L$ of a full rank.
Multiplying $L$ by an invertible $3\times 3$ matrix on the left and a permutation $4\times 4$ matrix on the right, and reordering $a_{j}$'s and $b_{j}$'s, we may assume that
\begin{equation}
L=
\begin{bmatrix}
1 & 0 & 0 & \lambda_1\\
0 & 1 & 0 & \lambda_2\\
0 & 0 & 1 & \lambda_3
\end{bmatrix}
\end{equation}
for some $\lambda_i$.
Then
\begin{equation}\label{180524e2.25}
L M_1 L^T=
\begin{bmatrix}
a_1 & 0 & 0\\
0 & a_2 & 0\\
0 & 0 & a_3
\end{bmatrix}
+ a_4 (\lambda_1, \lambda_2, \lambda_3)^T (\lambda_1, \lambda_2, \lambda_3)
\end{equation}
and
\begin{equation}\label{180524e2.26}
L M_2 L^T=
\begin{bmatrix}
b_1 & 0 & 0\\
0 & b_2 & 0\\
0 & 0 & b_3
\end{bmatrix}
+ b_4 (\lambda_1, \lambda_2, \lambda_3)^T (\lambda_1, \lambda_2, \lambda_3)
\end{equation}
Notice that
\begin{equation}\label{rank_triangle}
\rank (A+B)\le \rank (A)+\rank (B),
\end{equation}
for two arbitrary matrices and
\begin{equation}
\rank \Big((\lambda_1, \lambda_2, \lambda_3)^T (\lambda_1, \lambda_2, \lambda_3)\Big)\le 1.
\end{equation}
These two facts, combined with \eqref{181208e2.4}, imply that $a_1 a_2 a_3=0$ and $b_1 b_2 b_3=0$.
By \eqref{180713e1.9}, at most one of $a_{1},\dotsc,a_{4}$ can be $0$, so we may assume without loss of generality $a_{3}=0$.
Again by \eqref{180713e1.9}, at most one of $b_{1},\dotsc,b_{4}$ can be $0$, and if $a_{3}=0$ then $b_{3} \neq 0$, so we may assume without loss of generality $b_{2}=0$.
Two minor determinants of order $2\times 2$ of \eqref{180524e2.25} are
\begin{equation}
a_1 a_4 \lambda_3^2 \text{ and } a_2 a_4 \lambda_3^2.
\end{equation}
Hence, we must have $\lambda_3=0$, otherwise we have a contradiction to \eqref{181208e2.4}.
By a similar argument applied to $M_2$, we must have $\lambda_2=0$.
So far, we have obtained
\begin{equation}\label{180524e2.25a}
L M_1 L^T=
\begin{bmatrix}
a_1+a_4 \lambda_1^2 & 0 & 0\\
0 & a_2 & 0\\
0 & 0 & 0
\end{bmatrix}
\text{ and }
L M_2 L^T=
\begin{bmatrix}
b_1+b_4 \lambda_1^2 & 0 & 0\\
0 & 0 & 0\\
0 & 0 & b_3
\end{bmatrix}
\end{equation}
Since from \eqref{180713e1.9} we know $a_2\neq 0$ and $b_3\neq 0$, by \eqref{181208e2.4} we obtain
\begin{equation}
a_1+a_4 \lambda_1^2=b_1+b_4\lambda^2_1=0.
\end{equation}
This is a contradiction to \eqref{180713e1.9}.

\printbibliography

\end{document}